\RequirePackage[l2tabu, orthodox]{nag}

\documentclass[letterpaper, oneside]{article}

\usepackage[letterpaper]{geometry}

\usepackage[utf8]{inputenc}

\usepackage[british]{babel}

\usepackage[round]{natbib}

\usepackage[ntheorem]{empheq}
\usepackage{mathtools}

\usepackage{amsfonts, amssymb}

\usepackage[standard, amsmath, thref, thmmarks]{ntheorem}

\usepackage{algpseudocode}
\usepackage{algorithm}

\usepackage{booktabs}
\usepackage{dcolumn}

\usepackage{graphicx}
\graphicspath{{./img/}}

\usepackage[euler]{textgreek}

\newmuskip\Fmuskip%

\newcommand*{\Fcomma}{{\normalcomma}\mskip\Fmuskip}
\newcommand*{\F}[4][8]{%
  \begingroup%
  \Fmuskip=#1mu\relax%
  \mathchardef\normalcomma=\mathcode`,%
  \mathcode`\,=\string"8000%
  \begingroup\lccode`\~=`\,%
  \lowercase{\endgroup\let~}\Fcomma%
  {}_{2}F_{1}{\left[\genfrac..{0pt}{}{#2}{#3};#4\right]}%
  \endgroup%
}

\DeclareMathOperator{\sign}{sgn}

\newcolumntype{d}[1]{D{.}{.}{#1}}

\renewcommand\thealgorithm{\arabic{algorithm}}

\algnewcommand{\Input}{\item[\textbf{Input:}]}
\algnewcommand{\Output}{\item[\textbf{Output:}]}
\algnewcommand{\Initialization}[1]{\item[\textbf{Initialization:}]\Statex \parbox[t]{0.8\linewidth}{\raggedright #1}}
\algnewcommand{\LineComment}[1]{\hfill #1}

\begin{document}

\renewcommand*{\thefootnote}{\fnsymbol{footnote}}

\begin{flushleft}
{\Large\textbf{Numerical evaluation of the transition probability of the simple birth-and-death process}}
\newline
\\
Alberto Pessia\({}^{1,}\)\footnote{To whom correspondence should be addressed. Email: academic@albertopessia.com} and Jing Tang\({}^{1}\)
\\
\small{\({}^{1}\)Faculty of Medicine, University of Helsinki, Helsinki, Finland}
\end{flushleft}

\renewcommand*{\thefootnote}{\arabic{footnote}}
\setcounter{footnote}{0}

\begin{abstract}
The simple (linear) birth-and-death process is a widely used stochastic model for describing the dynamics of a population.
When the process is observed discretely over time, despite the large amount of literature on the subject, little is known about formal estimator properties.
Here we will show that its application to observed data is further complicated by the fact that numerical evaluation of the well-known transition probability is an ill-conditioned problem.
To overcome this difficulty we will rewrite the transition probability in terms of a Gaussian hypergeometric function and subsequently obtain a three-term recurrence relation for its accurate evaluation.
We will also study the properties of the hypergeometric function as a solution to the three-term recurrence relation.
We will then provide formulas for the gradient and Hessian of the log-likelihood function and conclude the article by applying our methods for numerically computing maximum likelihood estimates in both simulated and real dataset.
\end{abstract}

\section{Introduction}\label{sec:introduction}
A birth-and-death process (BDP) is a stochastic model that is commonly employed for describing changes over time of the size of a population.
Its first mathematical formulation is due to \citet{feller1939ab_grundlagen}, followed by important mathematical contributions of \citet{arley1944am_theory} and \citet{kendall1948ams_generalized, kendall1949jrsssbsm_stochastic}.
According to the basic assumptions of the model, when the population size at time \(t\) is \(j\), the probability of a single birth occurring during an infinitesimal time interval \((t, t + dt)\) is equal to \(\lambda_{j} dt + o(dt)\) while the probability of a single death is \(\mu_{j} dt + o(dt)\), where \(\lambda_{j} \geq 0\) and \(\mu_{j} \geq 0\) for \(j \geq 0\).
If \(p_{j}(t)\) is the probability of observing \(j\) individuals at time \(t\) then
\begin{equation*}
  p_{j}(t + dt) = \lambda_{j - 1} dt p_{j - 1}(t) + \mu_{j + 1} dt p_{j + 1}(t) + (1 - (\lambda_{j} + \mu_{j}) dt) p_{j}(t) + o(dt)
\end{equation*}
If we subtract \(p_{j}(t)\) from both sides of the equation, divide by \(dt\), and then take the limit of \(dt\) to zero, we obtain the well known BDP differential equation
\begin{equation}\label{eq:prob_bdp_1}
  p_{j}^{\prime}(t) = \lambda_{j - 1} p_{j - 1}(t) + \mu_{j + 1} p_{j + 1}(t) - (\lambda_{j} + \mu_{j}) p_{j}(t)
\end{equation}
By assuming that at time zero the size of the population was \(i \geq 0\), that is \(p_{i}(0) = 1\), we have the initial condition required to solve the differential equation (\ref{eq:prob_bdp_1}).

In this article we will focus on the simple (linear) BDP without migration \citep{kendall1949jrsssbsm_stochastic} defined by \(\lambda_{j} = j\lambda\) and \(\mu_{j} = j\mu\).
With this particular choice of parameters a starting size of zero implies \(\lambda_{0} = \mu_{0} = 0\), i.e. it remains zero for all \(t \geq 0\).
The rate of growth does not increase faster than the population size and therefore \(\sum_{j = 0}^{\infty} p_{j}(t) = 1\) \citep[Chapter 17, Section 4]{feller1968_introduction}.
When \(i > 0\) the population becomes extinct if it reaches the size \(j = 0\) at time \(t > 0\).
Obviously \(i\), \(j\), and \(t\) are not allowed to be negative, nor the basic birth and death rates.

What makes this model particularly attractive for real applications is the fact that its transition probability is available in closed form \citep[Chapter 8]{bailey1964_elements} and we could, in principle, easily evaluate it.
By defining
\begin{equation*}
  \begin{split}
    &\phi(t, \lambda, \mu) = \frac{e^{(\lambda - \mu) t} - 1}{\lambda e^{(\lambda - \mu) t} - \mu}, \qquad \alpha(t, \lambda, \mu) = \mu \phi(t, \lambda, \mu), \qquad \beta(t, \lambda, \mu) = \lambda \phi(t, \lambda, \mu)\\
    &\gamma(t, \lambda, \mu) = 1 - \alpha(t, \lambda, \mu) - \beta(t, \lambda, \mu) = 1 - (\lambda + \mu) \phi(t, \lambda, \mu)
  \end{split}
\end{equation*}
and assuming that at time 0 the size of the population was \(i > 0\), the probability of observing \(j\) units at time \(t\) is
\begin{empheq}[left={p_{j}(t)=\empheqlbrace}]{align}
  &\sum_{h = 0}^{m} \binom{i}{h} \binom{i + j - h - 1}{i - 1} \alpha(t, \lambda, \mu)^{i - h} \beta(t, \lambda, \mu)^{j - h} \gamma(t, \lambda, \mu)^{h}, &\mu \neq \lambda\label{eq:transprob_2}\\
  &\sum_{h = 0}^{m} \binom{i}{h} \binom{i + j - h - 1}{i - 1} \left(\frac{\lambda t}{1 + \lambda t}\right)^{i + j - 2 h} \left(\frac{1 - \lambda t}{1 + \lambda t}\right)^{h}, &\mu = \lambda\label{eq:transprob_3}\\
  &\binom{j - 1}{i - 1} e^{-j \lambda t} (e^{\lambda t} - 1)^{j - i}, &\mu = 0, j \geq i\label{eq:transprob_4}\\
  &\binom{i}{j} e^{-i \mu t} (e^{\mu t} - 1)^{i - j}, &\lambda = 0, j \leq i\label{eq:transprob_5}\\
  &1, &t = 0, j = i\label{eq:transprob_6}\\
  &1, &\lambda = \mu = 0, j = i\label{eq:transprob_7}\\
  &0, &\text{otherwise}\label{eq:transprob_8}
\end{empheq}
where \(t\), \(\lambda\), and \(\mu\) are to be considered strictly positive if not otherwise specified and \(m = \min(i, j)\).
The probability of the population being extinct at time \(t\) is
\begin{empheq}[left={p_{0}(t)=\empheqbiglbrace}]{align}
  &\left(\frac{\mu e^{(\lambda - \mu) t} - \mu}{\lambda e^{(\lambda - \mu) t} - \mu}\right)^{i}, &\mu \neq \lambda\label{eq:transprob_9}\\
  &\left(\frac{\lambda t}{1 + \lambda t}\right)^{i}, &\mu = \lambda\label{eq:transprob_10}\\
  &(1 - e^{-\mu t})^{i}, &\lambda = 0\label{eq:transprob_11}\\
  &0, &\mu = 0 \text{ or } t = 0\label{eq:transprob_12}
\end{empheq}

In the majority of applications direct evaluation of equations (\ref{eq:transprob_2})-(\ref{eq:transprob_12}) is sufficient.
However, for particular values of process parameters, equations (\ref{eq:transprob_2}) and (\ref{eq:transprob_3}) are numerically unstable (Figure \ref{fig:relative_error}) and alternative methods are needed.
\begin{figure}[t]
  \centering
  \includegraphics[width=\textwidth]{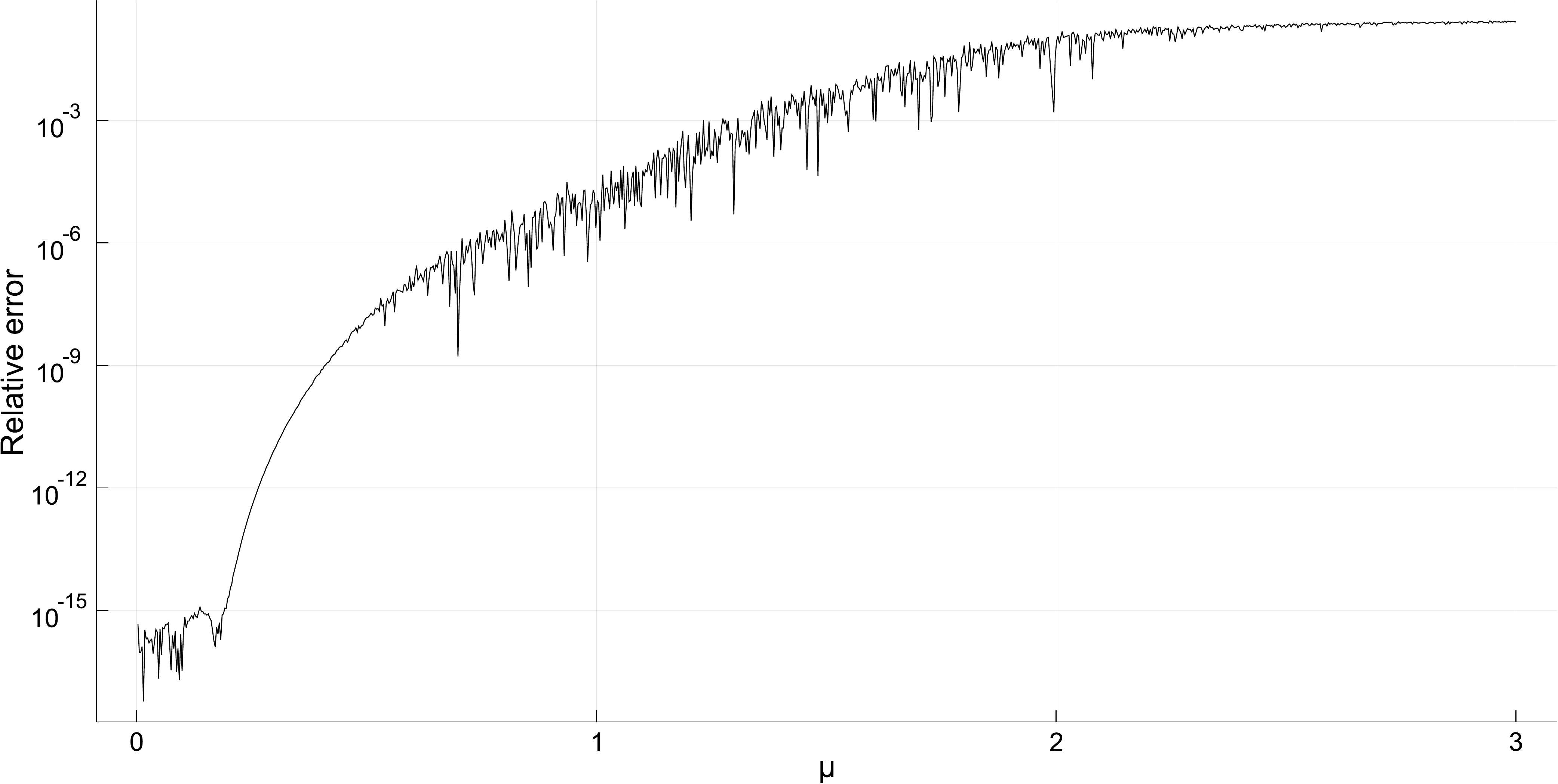}
  \caption{Numerical relative error of the log-probability as evaluated by direct application of equations (\ref{eq:transprob_2}) and (\ref{eq:transprob_3}). For this particular example we set \(i = 25\), \(j = 35\), \(t = 2\), \(\lambda = 1\) and evaluated the log-probability as a function of \(\mu\). We computed correct values in arbitrary precision with Maple\texttrademark{} 2018.2 \citep{monagan2005_maple}. Relative error is defined as \(|1 - \log \hat{p}_{j}(t) / \log p_{j}(t)|\) where \(\hat{p}_{j}(t)\) is the numerically evaluated transition probability.}\label{fig:relative_error}
\end{figure}
A possible approach could be the algorithm introduced by \citet{crawford2012jmb_transition} based on the continued fraction representation of \citet{murphy1975ijam_properties}, but for this particular case where we know the exact closed form we will show that a simpler and more efficient method is available.

The remainder of the paper is organized as follows.
In Section \ref{sec:numer_stabil} we introduce the problem and find the parameter sets for which it is ill-conditioned.
In Section \ref{sec:hyper_repr} we rewrite the transition probability in terms of a Gaussian hypergeometric function and find in Section \ref{sec:hyper_func} a three-term recurrence relation (TTRR) for its computation.
In Section \ref{sec:likelihood} we extend the results to the log-likelihood function, its gradient, and its Hessian matrix.
In Section \ref{sec:applications} we apply our method to both simulated and real data.
In Section \ref{sec:conclusions} we conclude the article with a brief discussion.

\section{Numerical stability}\label{sec:numer_stabil}
We will always assume that all basic mathematical operations (arithmetic, logarithmic, exponentiation, etc.) are computed with a relative error bounded by a value \(\epsilon\) that is close to zero and small enough for any practical application \citep{press2007_numerical}.
Following this assumption and after taking into consideration floating-point arithmetic \citep{goldberg1991acs_what}, equations (\ref{eq:transprob_4})-(\ref{eq:transprob_12}) can be considered numerically stable and won't be discussed further.
We will instead focus our attention on the series (\ref{eq:transprob_2}) and (\ref{eq:transprob_3}) assuming all variables to be strictly positive, including \(j\).

Suppose to be interested in the value \(s_{m} = \sum_{h=0}^{m} u_{h}\) and use a na{\"i}ve recursive summation algorithm for its computation, that is
\begin{equation*}
  \begin{split}
    s_{0} &= u_{0}\\
    s_{n} &= s_{n - 1} + u_{n}, \quad n = 1, \ldots, m
  \end{split}
\end{equation*}
The relative condition number of this algorithm is \citep{stummel1980fonc(na_rounding}
\begin{equation*}
  \rho_{m} = \rho_{m}^{A} + \rho_{m}^{R} = \frac{\sum\limits_{h=0}^{m} | u_{h} |}{\left|\sum\limits_{h=0}^{m} u_{h}\right|} + \frac{\sum\limits_{n=1}^{m} | s_{n} |}{\left|\sum\limits_{h=0}^{m} u_{h}\right|}
\end{equation*}
where \(\rho_{m}^{A}\) is associated with perturbations in the the value of the addends while \(\rho_{m}^{R}\) is associated with rounding errors in the arithmetic operations.
Note that when \(u_{h} \geq 0\) for all \(h\), \(\rho_{m}^{A} = 1\) and the condition number depends only on rounding errors.
With a compensated summation algorithm \citep{higham2002_accuracy} we might significantly reduce the numerical error and evaluate accurately the sum.
However, when the addends are alternating in sign, the condition number can be of large magnitude and the problem is numerically unstable even when compensating for rounding errors.
In our case it is likely that the magnitude of the binomial coefficients make \(\rho_{m}^{A}\) a ratio between a very large number and a probability that is instead close to zero.
We will now find the conditions under which the sums (\ref{eq:transprob_2}) and (\ref{eq:transprob_3}) are alternating in sign.

\begin{proposition}\label{prop:phi_fun}
For all \(\lambda \geq 0\) and \(\mu \geq 0\) the function
\begin{equation*}
  \phi(t, \lambda, \mu) = \frac{e^{(\lambda - \mu) t} - 1}{\lambda e^{(\lambda - \mu) t} - \mu}
\end{equation*}
is zero if and only if \(t = 0\). It is always positive otherwise.
\end{proposition}
\begin{proof}
If \(\mu \neq \lambda\) and \(t = 0\) the numerator \(e^{(\lambda - \mu) t} - 1\) is equal to zero but the denominator is not.
When \(\mu = \lambda\) the function becomes
\begin{equation*}
  \lim_{\mu \rightarrow \lambda} \frac{e^{(\lambda - \mu) t} - 1}{\lambda e^{(\lambda - \mu) t} - \mu} = \frac{t}{1 + \lambda t}
\end{equation*}
For all \(\lambda \geq 0\), it is equal to zero if \(t = 0\) and positive otherwise.
Assume now \(t > 0\).
When \(\lambda > \mu\) we have \(e^{(\lambda - \mu) t} > 1\) and \(\mu / \lambda < 1\) implying that the numerator and the denominator are always positive.
When \(\lambda < \mu\) the numerator \(e^{(\lambda - \mu) t} - 1\) is negative.
The denominator is also negative when \(e^{(\lambda - \mu) t} < \mu /\lambda\).
Since \(\lambda < \mu\) the left hand side is less than one while the right hand side is greater than one, proving the proposition.
\end{proof}

\begin{corollary}\label{corol:nonneg}
Functions \(\alpha(t, \lambda, \mu) = \mu \phi(t, \lambda, \mu)\) and \(\beta(t, \lambda, \mu) = \lambda \phi(t, \lambda, \mu)\) are non-negative for all \(\lambda \geq 0\), \(\mu \geq 0\), and \(t \geq 0\).
\end{corollary}
\begin{proof}
This is a direct consequence of Proposition (\ref{prop:phi_fun}) and the assumptions \(\lambda \geq 0\) and \(\mu \geq 0\).
\end{proof}

\begin{proposition}\label{prop:gamma_fun}
Assume \(t > 0\), \(\lambda > 0\), and \(\mu > 0\).
Let \(\gamma(t, \lambda, \mu) = 1 - (\lambda + \mu) \phi(t, \lambda, \mu)\).
If \(\mu \neq \lambda\), define \(\xi = \log(\lambda / \mu) / (\lambda - \mu)\).
If \(\mu = \lambda\), define instead \(\xi = 1 / \lambda\).
Then
\begin{empheq}[left={\gamma(t, \lambda, \mu) \empheqbiglbrace}]{align*}
  &< 0 \text{ if } t > \xi\\
  &> 0 \text{ if } t < \xi\\
  &= 0 \text{ if } t = \xi
\end{empheq}
\end{proposition}
\begin{proof}
Rewrite the function \(\gamma(t, \lambda, \mu)\) in its expanded form:
\begin{equation*}
  \gamma(t, \lambda, \mu) = \frac{\lambda - \mu e^{(\lambda - \mu) t}}{\lambda e^{(\lambda - \mu) t} - \mu}
\end{equation*}
Assume \(\mu \neq \lambda\).
If \(t > 0\), we already proved in Proposition (\ref{prop:phi_fun}) that the denominator is always positive when \(\lambda > \mu\) while it is always negative when \(\lambda < \mu\).
The numerator is positive when \((\lambda - \mu) t < \log(\lambda / \mu)\), that is when \(t < \xi\) and \(\lambda > \mu\) or \(t > \xi\) and \(\lambda < \mu\).
It is zero when \(t = \xi\).
From these results follow the first set of inequalities.
When \(\mu = \lambda\) the function becomes
\begin{equation*}
  \lim_{\mu \rightarrow \lambda} \frac{\lambda - \mu e^{(\lambda - \mu) t}}{\lambda e^{(\lambda - \mu) t} - \mu} = \frac{1 - \lambda t}{1 + \lambda t}
\end{equation*}
If \(t > 0\) and \(\lambda > 0\) the denominator is always positive.
The numerator is zero if \(t = \lambda^{-1}\), it is positive when \(t < \lambda^{-1}\), and it is negative otherwise.
\end{proof}

\begin{corollary}\label{corol:zeroterm}
When \(t = \log(\lambda / \mu) / (\lambda - \mu)\) equation (\ref{eq:transprob_2}) becomes
\begin{equation*}
  p_{j}(t) = \binom{i + j - 1}{i - 1} \left(\frac{\mu}{\lambda + \mu}\right)^{i} \left(\frac{\lambda}{\lambda + \mu}\right)^{j}
\end{equation*}
When \(t = \lambda^{-1}\) equation (\ref{eq:transprob_3}) becomes
\begin{equation*}
  p_{j}(t) = \binom{i + j - 1}{i - 1} \left(\frac{1}{2}\right)^{i + j}
\end{equation*}
\end{corollary}
\begin{proof}
This is a direct consequence of the fact that \(\gamma(t, \lambda, \mu) = 0\) and that \(0^{h}\) is zero for all \(h > 0\) and one for \(h = 0\).
\end{proof}

From Corollary \ref{corol:zeroterm} we have simple closed form solutions when \(\gamma(t, \lambda, \mu)\) is zero and therefore we will not consider this case further.
We know from Proposition \ref{prop:gamma_fun} the conditions under which equations (\ref{eq:transprob_2}) and (\ref{eq:transprob_3}) are alternating in sign and might become numerically unstable.
Looking at the example shown in Figure \ref{fig:relative_error}, where \(t = 2\) and \(\lambda = 1\), function \(\gamma(t, \lambda, \mu)\) is non-negative when \(0 < \mu \lesssim 0.2032\).
We clearly see from Figure \ref{fig:relative_error} that the error steadily increases starting at the value \(\mu \approx 0.2032\).
In the next section we will find an alternative representation to equations (\ref{eq:transprob_2}) and (\ref{eq:transprob_3}) that will lead to an algorithm for their accurate evaluation.

\section{Hypergeometric representation}\label{sec:hyper_repr}
Define
\begin{align*}
  \omega(i, j, t, \lambda, \mu) &= \binom{i + j - 1}{i - 1} \alpha(t, \lambda, \mu)^{i} \beta(t, \lambda, \mu)^{j} = \binom{i + j - 1}{i - 1} \mu^{i} \lambda^{j} \left(\frac{e^{(\lambda - \mu) t} - 1}{\lambda e^{(\lambda - \mu) t} - \mu}\right)^{i + j}\\
  z(t, \lambda, \mu) &= \frac{\gamma(t, \lambda, \mu)}{\alpha(t, \lambda, \mu) \beta(t, \lambda, \mu)} = \frac{(\lambda - \mu e^{(\lambda - \mu) t}) (\lambda e^{(\lambda - \mu) t} - \mu)}{\lambda \mu (e^{(\lambda - \mu) t} - 1)^{2}}
\end{align*}
Note that \(\omega(i, j, t, \lambda, \mu)\) is simply the first term in the summation (\ref{eq:transprob_2}).
When \(\mu = \lambda\) set
\begin{align*}
  \omega(i, j, t, \lambda, \lambda) &= \lim_{\mu \rightarrow \lambda} \omega(i, j, t, \lambda, \mu) = \binom{i + j - 1}{i - 1} \left(\frac{\lambda t}{1 + \lambda t}\right)^{i + j}\\
  z(t, \lambda, \lambda) &= \lim_{\mu \rightarrow \lambda} z(t, \lambda, \mu) = \left(\frac{1}{\lambda t}\right)^{2} - 1
\end{align*}
Multiply and divide each term in the series (\ref{eq:transprob_2}) by \(\omega(i, j, t, \lambda, \mu)\) to get
\begin{equation}\label{eq:transprob_hyper}
  \begin{split}
    p_{j}(t) &= \omega(i, j, t, \lambda, \mu) \sum_{h = 0}^{m} \frac{\binom{i}{h} \binom{i + j - h - 1}{i - 1}}{\binom{i + j - 1}{i - 1}} z(t, \lambda, \mu)^{h} = \omega(i, j, t, \lambda, \mu) \sum_{h = 0}^{m} \frac{\binom{i}{h} \binom{j}{h}}{\binom{i + j - 1}{h}} z(t, \lambda, \mu)^{h} =\\
    &= \omega(i, j, t, \lambda, \mu) \sum_{h = 0}^{m} \frac{i!}{(i - h)!} \frac{j!}{(j - h)!} \frac{(i + j - 1 - h)!}{(i + j - 1)!} \frac{z(t, \lambda, \mu)^{h}}{h!} =\\
    &= \omega(i, j, t, \lambda, \mu) \sum_{h = 0}^{m} \frac{(-i)_{h} (-j)_{h}}{(-(i + j - 1))_{h}} \frac{(-z(t, \lambda, \mu))^{h}}{h!} =\\
    &= \omega(i, j, t, \lambda, \mu) \, \F{-i, -j}{-(i + j - 1)}{-z(t, \lambda, \mu)}
  \end{split}
\end{equation}
where \((q)_{h}\) is the rising Pochhammer symbol and \({}_{2}F_{1}(a, b; c; y)\) is the Gaussian hypergeometric function \citep[Chapter 1]{slater1966_generalized}.
To evaluate (\ref{eq:transprob_hyper}) is then sufficient to separately compute the functions \(\omega(i, j, t, \lambda, \mu)\) and \({}_{2}F_{1}(-i, -j; -(i + j - 1); -z(t, \lambda, \mu))\).

Partial derivatives of \(\log p_{j}(t)\) are required for computing partial derivatives of the log-likelihood function as we will explicitly show in Section \ref{sec:likelihood}.
The following theorem proves that partial derivatives of \({}_{2}F_{1}(-i, -j; -(i + j - 1); -z(t, \lambda, \mu))\) depend on hypergeometric functions of similar nature.
\begin{theorem}\label{theorem:derivative}
Denote with \(u_{x}(x, y)\) the first-order partial derivative of \(u(x, y)\) with respect to \(x\).
Similarly, denote with \(u_{xy}(x, y)\) the second-order partial derivative with respect first to \(x\) and subsequently \(y\).
Then
\begin{equation*}
  \frac{\partial}{\partial x} \F[1]{-i, -j}{-(i + j - 1)}{-u(x, y)} = \frac{i j}{i + j - 1} u_{x}(x, y) \F{-(i - 1), -(j - 1)}{-(i + j - 2)}{-u(x, y)}
\end{equation*}
\begin{multline*}
  \frac{\partial^{2}}{\partial x \partial y} \F[1]{-i, -j}{-(i + j - 1)}{-u(x, y)} = \frac{i j}{i + j - 1} u_{xy}(x, y) \F{-(i - 1), -(j - 1)}{-(i + j - 2)}{-u(x, y)} +\\
  + \frac{i (i - 1) j (j - 1)}{(i + j - 1) (i + j - 2)} u_{x}(x, y) u_{y}(x, y) \F{-(i - 2), -(j - 2)}{-(i + j - 3)}{-u(x, y)}
\end{multline*}
\end{theorem}
\begin{proof}
\begin{equation*}
  \begin{split}
    &\frac{\partial}{\partial x} \left(1 + \sum_{h = 1}^{m} \frac{i!}{(i - h)!} \frac{j!}{(j - h)!} \frac{(i + j - 1 - h)!}{(i + j - 1)!} \frac{u(x, y)^{h}}{h!}\right) =\\
    &=u_{x}(x, y) \sum_{h = 1}^{m} \frac{i!}{(i - h)!} \frac{j!}{(j - h)!} \frac{(i + j - 1 - h)!}{(i + j - 1)!} \frac{u(x, y)^{h - 1}}{(h - 1)!} =\\
    &= \frac{i j}{i + j - 1} u_{x}(x, y) \sum_{h = 0}^{m - 1} \frac{(i - 1)!}{(i - 1 - h)!} \frac{(j - 1)!}{(j - 1 - h)!} \frac{(i + j - 2 - h)!}{(i + j - 2)!} \frac{u(x, y)^{h}}{h!} =\\
    &= \frac{i j}{i + j - 1} u_{x}(x, y) \F{-(i - 1), -(j - 1)}{-(i + j - 2)}{-u(x, y)}
  \end{split}
\end{equation*}
where \(m = \min(i, j)\).
Apply the same procedure to obtain the second-order partial derivatives.
\end{proof}

With the substitutions \(a_{1} = i - 1\) and \(b_{1} = j - 1\) the hypergeometric function in the first-order partial derivatives becomes \({}_{2}F_{1}(-a_{1}, -b_{1}; - (a_{1} + b_{1}); -z(t, \lambda, \mu))\).
Similarly, with the substitutions \(a_{2} = i - 2\) and \(b_{2} = j - 2\), the hypergeometric function in the second-order partial derivatives becomes \({}_{2}F_{1}(-a_{2}, -b_{2}; - (a_{2} + b_{2} + 1); -z(t, \lambda, \mu))\).
In general, we must be able to accurately evaluate the hypergeometric function
\begin{equation}\label{eq:hypergeometric_function}
  \F{-a, -b}{-(a + b - k)}{-z} = \sum_{h = 0}^{\min(a, b)} \frac{\binom{a}{h} \binom{b}{h}}{\binom{a + b - k}{h}} z^{h}
\end{equation}
for \(a, b \in \mathbb{N}_{+}\), \(k = 1, 0, -1, -2, \ldots\), and \(z \in \mathbb{R}\).

\section{Hypergeometric function \texorpdfstring{\({}_{2}F_{1}(-a, -b; -(a + b - k); -z)\)}{2F1(-a, -b; -(a + b - k); -z)}}\label{sec:hyper_func}
The following theorem can be considered the main result of this article.
\begin{theorem}\label{theorem:recurrence}
The hypergeometric function \({}_{2}F_{1}(-a, -b; -(a + b - k); -z)\), as a function of \(b\), is a solution of the three-term recurrence relation (TTRR)
\begin{equation}\label{eq:recurrence}
  (a + b + 1 - k) (a + b - k) y_{b + 1} - (a + b - k) (a + b + 1 - k + (a - b) z) y_{b} - b (b - k) z y_{b - 1} = 0
\end{equation}
\end{theorem}
\begin{proof}
The recursion can be obtained by the method of creative telescoping \citep{petkovsek1996_aequalb, zeilberger1991jsc_method}.
To prove that it holds, define
\begin{equation*}
  L_{b, h} = \frac{a!}{(a - h)!} \frac{b!}{(b - h)!} \frac{(a + b - k - h)!}{(a + b - k)!} \frac{z^{h}}{h!}
\end{equation*}
and let
\begin{equation*}
  t_h = (a + b + 1 - k) (a + b - k) L_{b + 1, h} - (a + b - k) (a + b + 1 - k + (a - b) z) L_{b, h} - b (b - k) z L_{b - 1, h}
\end{equation*}
Note that \(\sum_{h} t_{h}\) is the left hand side of equation (\ref{eq:recurrence}) because \(y_{b} = \sum_{h} L_{b, h}\).
Set
\begin{equation*}
  R_{b, h} = - \frac{(a - k) (a + b - k - h) b h}{(b + 1 - h) (b - h)}
\end{equation*}
and let
\begin{equation*}
  u_{h} = R_{b, h + 1} L_{b - 1, h + 1} - R_{b, h} L_{b - 1, h}
\end{equation*}
Sum the previous expression with respect to \(h\) to obtain \(\sum_{h} u_{h} = - R_{b, 0} L_{b - 1, 0} = 0\).
We now need to prove that \(t_{h} = u_{h}\) for all \(h\).
Start by dividing \(t_{h}\) by \(L_{b - 1, h}\) to obtain
\begin{multline*}
  \frac{(a + b + 1 - k - h) (a + b - k - h) (b + 1) b}{(b + 1 - h) (b - h)} +\\
  - \frac{b (a + b - k - h) (a + b + 1 - k + (a - b) z)}{b - h} - b (b - k) z
\end{multline*}
By expanding the polynomial and collecting the terms with respect to \(h\) we get
\begin{equation*}
  \frac{t_h}{L_{b - 1, h}} = -\frac{(a - k) b}{(b + 1 - h) (b - h)} \left[(1 + z) h^{2} - (a + b - k + (a + b + 1) z) h + a (b + 1) z \right]
\end{equation*}
Doing the same with the right hand side we get
\begin{equation*}
  \begin{split}
    \frac{u_{h}}{L_{b - 1, h}} &= -\frac{(a - k) b}{(b + 1 - h) (b - h)} \left[(b + 1 - h ) (a - h) z - h (a + b - k - h)\right] =\\
    &= -\frac{(a - k) b}{(b + 1 - h) (b - h)} \left[(1 + z) h^{2} - (a + b - k + (a + b + 1) z) h + a (b + 1) z \right]
  \end{split}
\end{equation*}
proving the equality.
\end{proof}

If we divide both sides of equation (\ref{eq:recurrence}) by the coefficient of \(y_{b + 1}\), and shift the index by 1, we obtain the \textit{forward} recursion
\begin{equation}\label{eq:forward_recursion}
  y_{b} = \left(1 + \frac{(a + 1 - b) z}{a + b - k}\right) y_{b - 1} + \frac{(b - 1) (b - 1 - k) z}{(a + b - k) (a + b - 1 - k)} y_{b - 2}
\end{equation}
Starting from
\begin{equation*}
  \begin{split}
    y_{0} &= {}_{2}F_{1}(-a, 0; -(a - k); -z) = 1\\
    y_{1} &= {}_{2}F_{1}(-a, 1; -(a + 1 - k); -z) = 1 + \frac{a z}{a + 1 - k}
  \end{split}
\end{equation*}
we can, in principle, obtain all remaining solutions all the way up to the required term.
Note that \(a \geq 1\) and \(k \leq 1\), therefore the denominator \(a + 1 - k\) is strictly positive and always well defined.
Knowing the values of \(y_{b + 2}\) and \(y_{b + 1}\), for large \(b\), we can travel the recursion also in a \textit{backward} manner:
\begin{equation}\label{eq:backward_recursion}
  y_{b} = \frac{(a + b + 2 - k) (a + b + 1 - k)}{(b + 1) (b + 1 - k) z} \left(y_{b + 2} - \left(1 + \frac{(a - b - 1) z}{a + b + 2 - k}\right) y_{b + 1}\right)
\end{equation}

Theorem \ref{theorem:recurrence} proves that the hypergeometric function \({}_{2}F_{1}(-a, -b; -(a + b - k); -z)\) is a solution of a TTRR.
However, equation (\ref{eq:recurrence}) can also admit a second linearly independent solution.

\begin{definition}\label{def:minimal_dominant}
A solution \(f_{b}\) of a TTRR is said to be a \textit{minimal} solution if there exists a linearly independent solution \(g_{b}\) such that
\begin{equation*}
  \lim_{b \rightarrow \infty} \frac{f_{b}}{g_{b}} = 0
\end{equation*}
The solution \(g_{b}\) is called a \textit{dominant} solution.
\end{definition}

It is well known that, regardless of the starting values, forward evaluation of a TTRR converges to the dominant solution while backward evaluation converges instead to the minimal solution \citep[Chapter 4]{gil2007_numerical}.
We now need to find the conditions under which our hypergeometric function is either the dominant or the minimal solution.

\begin{lemma}[Poincaré-Perron]\label{lemma:poinperro}
Let \(y_{b + 1} + v_{b} y_{b} + u_{b} y_{b - 1} = 0\) and suppose that \(v_{b}\) and \(u_{b}\) are different from zero for all \(b > 0\).
Suppose also that \(\lim_{b \rightarrow \infty} v_{b} = v\) and \(\lim_{b \rightarrow \infty} u_{b} = u\).
Denote with \(\phi_{1}\) and \(\phi_{2}\) the (not necessarily distinct) roots of the characteristic equation \(\phi^{2} + v \phi + u = 0\).
If \(f_{b}\) and \(g_{b}\) are the linearly independent non-trivial solutions of the difference equation, then
\begin{equation*}
  \limsup_{b \rightarrow \infty} \sqrt[b]{|f_{b}|} = |\phi_{1}|, \quad \limsup_{b \rightarrow \infty} \sqrt[b]{|g_{b}|} = |\phi_{2}|
\end{equation*}
If \(|\phi_{1}| < |\phi_{2}|\) it is also
\begin{equation*}
  \lim_{b \rightarrow \infty} \frac{f_{b + 1}}{f_{b}} = \phi_{1}, \quad \lim_{b \rightarrow \infty} \frac{g_{b + 1}}{g_{b}} = \phi_{2}
\end{equation*}
and \(f_{b}\) is the minimal solution while \(g_{b}\) is the dominant solution.
If \(|\phi_{1}| = |\phi_{2}|\) the lemma is inconclusive about the existence of a minimal solution.
\end{lemma}
\begin{proof}
See Chapter 8 of \citet{elaydi2005_introduction}.
\end{proof}

Using Lemma \ref{lemma:poinperro} we can study the nature of our hypergeometric function as a solution of the TTRR.

\begin{theorem}\label{theorem:minimal_dominant}
\({}_{2}F_{1}(-a, -b; -(a + b - k); -z)\) is a dominant solution of equation (\ref{eq:recurrence}) when \(|z| < 1\).
It is a minimal solution when \(|z| > 1\).
The nature of the solution is unknown when \(|z| = 1\).
\end{theorem}
\begin{proof}
Our TTRR is
\begin{equation*}
  y_{b + 1} - \left(1 + \frac{(a - b) z}{a + b + 1 - k}\right) y_{b} - \frac{b (b - k) z}{(a + b + 1 - k) (a + b - k)} y_{b - 1} = 0
\end{equation*}
Take the limit of the coefficients
\begin{equation*}
  \lim_{b \rightarrow \infty} - \left(1 + \frac{a - b}{a + b + 1 - k} z\right) = - (1 - z), \quad \lim_{b \rightarrow \infty} - \frac{b (b - k) z}{(a + b + 1 - k) (a + b - k)} = -z
\end{equation*}
The characteristic equation is \(\phi^{2} - (1 - z) \phi - z = 0\) with solutions \(\phi_{1} = -z\) and \(\phi_{2} = 1\).
When \(|z| < 1\) the solution associated with \(\phi_{1}\) is minimal and the one associated with \(\phi_{2}\) is dominant.
The opposite is true when \(|z| > 1\).
We will now prove that \({}_{2}F_{1}(-a, -b; -(a + b - k); -z)\) is associated with the characteristic root \(\phi_{2} = 1\).
The summation index \(h\) in equation (\ref{eq:hypergeometric_function}) depends on \(b\) since the upper bound of the series is the minimum between \(a\) and \(b\).
Note, however, that variable \(a\) is considered known and fixed to a finite value.
When \(b\) goes to infinity the summation index \(h\) in (\ref{eq:hypergeometric_function}) does not depend on \(b\) any more and the series is always finite, so that we can safely exchange the limit of the sum with the sum of the limits:
\begin{equation*}
  \lim_{b \rightarrow \infty} \sum_{h = 0}^{\min(a, b)} \frac{(-a)_{h} (-b)_{h}}{(-(a + b - k))_{h}} \frac{(-z)^{h}}{h!} = \sum_{h = 0}^{a} (-a)_{h} \frac{(-z)^{h}}{h!} \lim_{b \rightarrow \infty} \frac{(-b)_{h}}{(-(a + b - k))_{h}}
\end{equation*}
Using Stirling's approximation \(n! \sim \sqrt{2 \pi n} (n / e)^{n}\) for large \(n\), we obtain
\begin{equation*}
  \begin{split}
    &\lim_{b \rightarrow \infty} \frac{(-b)_{h}}{(-(a + b - k))_{h}} = \lim_{b \rightarrow \infty} \frac{b^{b + 1/2} (a + b - k - h)^{a + b - k - h + 1/2}}{(b - h)^{b - h + 1/2} (a + b - k)^{a + b - k + 1/2}} =\\
    = &\lim_{b \rightarrow \infty} \left(\frac{b}{b - h}\right)^{b + 1/2} \left(\frac{a + b - k - h}{a + b - k}\right)^{b - k + 1/2} \left(\frac{a + b - k - h}{a + b - k}\right)^{a} \left(\frac{b - h}{a + b - k - h}\right)^{h} =\\
    = &\lim_{b \rightarrow \infty} \left(\frac{b}{b - h}\right)^{b + 1/2} \left(\frac{a + b - k - h}{a + b - k}\right)^{b - k + 1/2} = e^{h} e^{-h} = 1
  \end{split}
\end{equation*}
from which follows that
\begin{equation*}
  \lim_{b \rightarrow \infty} \F{-a, -b}{-(a + b - k)}{-z} = \sum_{h = 0}^{a} (-a)_{h} \frac{(-z)^{h}}{h!} = \sum_{h = 0}^{a} \binom{a}{h} (-z)^{h} = (1 - z)^{a}
\end{equation*}
and
\begin{equation*}
  \lim_{b \rightarrow \infty} \frac{\F{-a, -(b + 1)}{-(a + b + 1 - k)}{-z}}{\F{-a, -b}{-(a + b - k)}{-z}} = \frac{(1 - z)^{a}}{(1 - z)^{a}} = 1
\end{equation*}
The solution is therefore dominant when \(|z| < 1\) and minimal for \(|z| > 1\).
When \(|z| = 1\) Lemma \ref{lemma:poinperro} is inconclusive about the nature of the solution.
\end{proof}

Since Lemma \ref{lemma:poinperro} refers to asymptotic results, Theorem \ref{theorem:minimal_dominant} is always valid for large values of \(b\).
For small values of \(b\), instead, there is a possibility of anomalous behaviour as described by \citet{deano2007mc_transitory}.
By Definition \ref{def:minimal_dominant} we would expect that the sequence of ratios between a minimal and a dominant solution would be monotonically decreasing to zero for all \(b\).
There are cases, however, in which this is not necessarily true.
A minimal solution might behave as a dominant solution up to a finite value \(b^{*}\) and then switch to its asymptotic minimal nature only starting at \(b^{*} + 1\).

\begin{definition}
Let \(f_{b}\) and \(g_{b}\) be respectively the minimal and dominant solution of a TTRR as \(b \rightarrow \infty\).
\(f_{b}\) is said to be \textit{pseudo-dominant} for all \(b \leq b^{*}\) if the sequence \(\{R_{b} = |f_{b} / g_{b}|\}\) is increasing for \(b \leq b^{*}\) but decreasing for \(b > b^{*}\).
\end{definition}

\begin{lemma}[Dea{\~n}o-Segura]\label{lemma:pseudominimal}
Let \(y_{b + 1} + v_{b} y_{b} + u_{b} y_{b - 1} = 0\) be a recurrence such that, for \(b \geq b^{-}\), \(u_{b} < 0\) and \(v_{b}\) changes sign at \(b^{*} > b^{-} + 1\).
Suppose that there exists a solution \(f_{b}\) with fixed pattern of signs for all \(b \geq b^{-}\), the pattern being alternating if \(v_{b} < 0\) for large \(b\) or with constant sign if \(v_{b} > 0\) for large \(b\) (\(f_{b}\) may be minimal).
Let \(g_{b}\) be any solution (not minimal) such that
\begin{equation*}
  \frac{g_{b^{*} + 1}}{g_{b^{*}}} = - \psi \frac{f_{b^{*} + 1}}{f_{b^{*}}}, \qquad \psi > 0,
\end{equation*}
and let \(R_{b} = |f_{b} / g_{b}|\), then for \(b \geq b^{-}\) the following holds depending on the value \(\psi\):
\begin{itemize}
  \item[(1)] if \(\psi > 1\), then \(R_{b} < R_{b^{*}}\) if \(b \neq b^{*}\).
  \item[(2)] if \(\psi < 1\), then \(R_{b} < R_{b^{*} + 1}\) if \(b \neq b^{*} + 1\).
  \item[(3)] if \(\psi = 1\), then \(R_{b} < R_{b^{*}} = R_{b^{*} + 1}\) if \(b \neq b^{*}, b^{*} + 1\).
\end{itemize}
\end{lemma}
\begin{proof}
See \citet{deano2007mc_transitory}.
\end{proof}

According to Lemma \ref{lemma:pseudominimal}, if \(u_{b}\) is negative and \(v_{b}\) changes sign at index \(b^{*}\), then our asymptotic minimal solution behaves as a dominant solution up to \(b^{*} - 1\) or \(b^{*}\).
We must then study the sign of the two coefficients
\begin{empheq}{align*}
  u_{b} &= - \frac{b (b - k) z}{(a + b + 1 - k) (a + b - k)}\\
  v_{b} &= -\frac{(a + b + 1 - k) + (a - b) z}{a + b + 1 - k}
\end{empheq}
with \(b \geq 1\), \(a \geq 1\) and \(k \leq 1\).
Since the denominators are strictly positive, we can simply study the signs of the associated quantities
\begin{empheq}{align*}
  u_{b}^{\prime} &= - b (b - k) z\\
  v_{b}^{\prime} &= - (a + b + 1 - k) - (a - b) z
\end{empheq}
\(u_{b}^{\prime}\) is negative when \(z > 0\), positive when \(z < 0\), and zero when \(b = k = 1\).
Define \(b^{*} = (z - 1)^{-1} ((z + 1) a + 1 - k)\).
\(v_{b}^{\prime}\) is negative when \(z < 1\) and \(b > b^{*}\) or when \(z > 1\) and \(b < b^{*}\).
It is obviously positive in the complementary set.
The point \(b^{*}\) is the delimiter at which the coefficient \(v_{b}\) switches from positive sign to negative sign or vice versa.

When \(z > 1\) we are under the conditions of Lemma \ref{lemma:pseudominimal}, therefore the solution is surely minimal for \(b > b^{*} + 1\).
It is pseudo-dominant for all \(b < b^{*}\).
Not knowing the shape of the linearly independent solution \(g_{b}\), we don't know if the solution becomes minimal at \(b^{*}\) or \(b^{*} + 1\).
Interestingly, when \(z < - (a + 2 - k) / (a - 1)\), we have the opposite behaviour of a positive \(u_{b}\) and \(v_{b}\) changing sign from positive to negative at the same index \(b^{*}\).
The regions are highlighted in Figure \ref{fig:coefficient_sign}.
\begin{figure}[t]
  \centering
  \includegraphics[width=\textwidth]{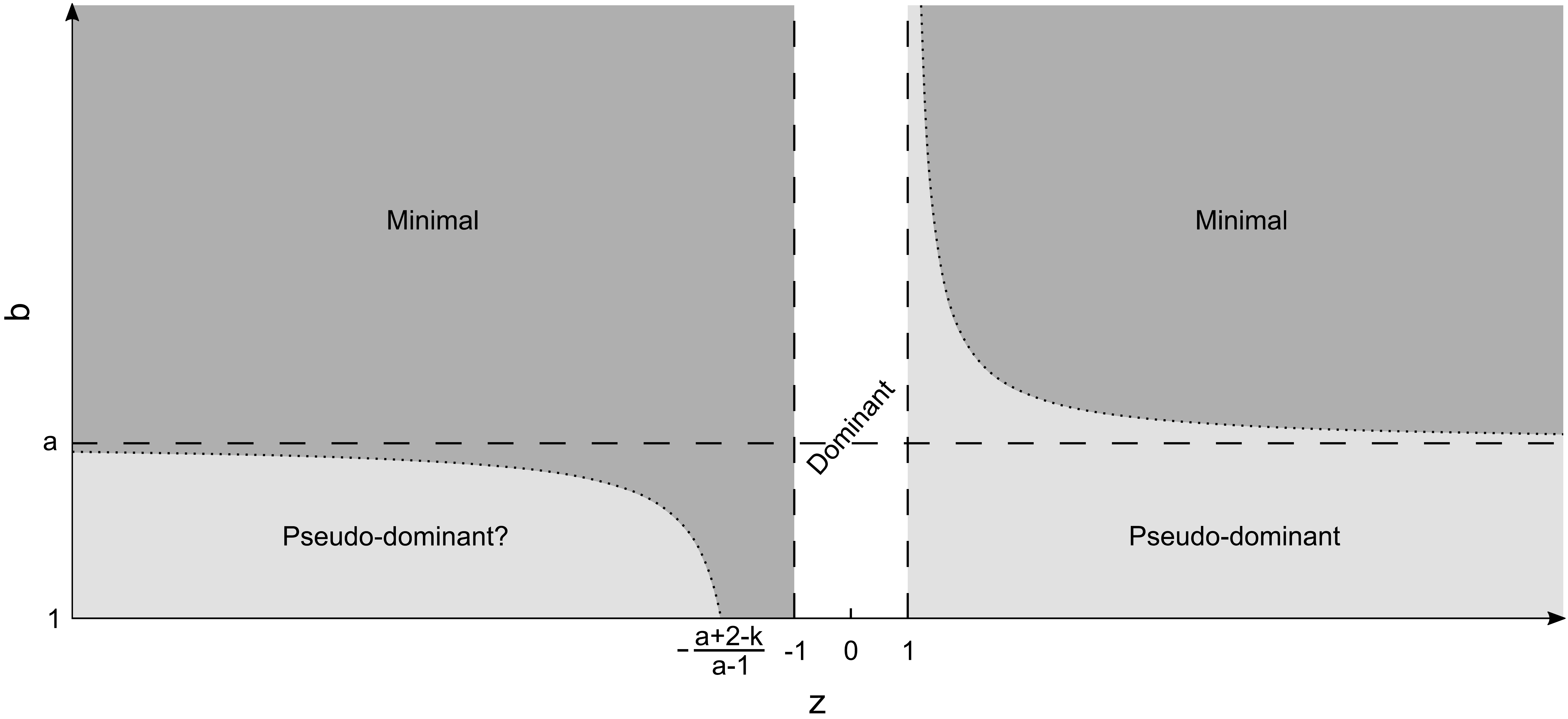}
  \caption{Nature of the hypergeometric function \({}_{2}F_{1}(-a, -b; -(a + b - k); -z)\) as a solution of TTRR (\ref{eq:recurrence}). The minimum value admissible for \(b\) is one. When \(b = 0\) we can simply use the numerically stable equations (\ref{eq:transprob_9})-(\ref{eq:transprob_12}). The dotted curve is given by equation \(b^{*} = (z - 1)^{-1} ((z + 1) a + 1 - k)\). It is dotted to represent the fact that we don't know if the solution becomes minimal at \(b^{*}\) or \(b^{*} + 1\). The curve has an horizontal asymptote at \(b = a\).}\label{fig:coefficient_sign}
\end{figure}
Lemma \ref{lemma:pseudominimal} does not consider the case of a positive \(u_{b}\) but we conjecture that it might be applied to this case as well.
Nevertheless, as shown by the following proposition, we can simply ignore the problem altogether.

\begin{proposition}\label{prop:z}
For all finite \(\lambda > 0\), \(\mu > 0\), and \(t > 0\), function \(z(t, \lambda, \mu)\) is always greater than -1.
It is positive when \(\mu \neq \lambda\) and \(t < \log(\lambda / \mu) / (\lambda - \mu)\) or when \(\mu = \lambda\) and \(t < \lambda^{-1}\).
\end{proposition}
\begin{proof}
Rewrite the function \(z(t, \lambda, \mu)\) as
\begin{equation*}
  z(t, \lambda, \mu) = \frac{\gamma(t, \lambda, \mu)}{\alpha(t, \lambda, \mu) \beta(t, \lambda, \mu)} = \frac{(\frac{\lambda}{\mu} + \frac{\mu}{\lambda}) e^{(\lambda - \mu) t} - e^{2 (\lambda - \mu) t} - 1}{e^{2 (\lambda - \mu) t} - 2 e^{(\lambda - \mu) t} + 1}
\end{equation*}
It is straightforward to show that the function converges to \(-1\) when \(\lambda\), \(\mu\), or \(t\) go to infinity.
The limit is never attained for finite \(\lambda\), \(\mu\), or \(t\).
When any of the parameters approaches zero, instead, the function approaches positive infinity.
We know from Corollary \ref{corol:nonneg} that the denominator \(\alpha(t, \lambda, \mu) \beta(t, \lambda, \mu)\) is always positive.
The sign of the function \(z(t, \lambda, \mu)\) is therefore equal to the sign of \(\gamma(t, \lambda, \mu)\), which is given in Proposition \ref{prop:gamma_fun}.
Same results apply when \(\mu = \lambda\).
\end{proof}

Note that for \(z > 1\), as clearly shown in Figure \ref{fig:coefficient_sign}, we have to use either the forward or backward recursion depending on the value of \(b\) that we wish to evaluate.
We can simplify our computations by applying the well known symmetric property \({}_{2}F_{1}(-a, -b; -(a + b - k); -z) = {}_{2}F_{1}(-b, -a; -(a + b - k); -z)\).
If \(b > a\), swap the two variables to transform a minimal solution into a pseudo-dominant one.
Using this trick we can apply the forward recursion for all \(z > -1\).

All the previous results are summarized in Algorithm \ref{alg:hyper_eval} in Appendix \ref{appendix:algorithm}.
Assuming a constant time for arithmetic operations the time complexity is simply \(O(m)\), where \(m = \min(a, b)\), that is the total number of iterations required.
Note that we only use basic arithmetic operations, saving computational time when compared to the more expensive functions found in equations (\ref{eq:transprob_2})-(\ref{eq:transprob_3}), such as the Binomial/Gamma.
Using the TTRR approach is better, from a computationally point of view, also when the problem is well-behaved.

\section{Likelihood function}\label{sec:likelihood}
Let \(\mathbf{t} = (t_{0}, \ldots, t_{S})^{T}\) be the vector of observation times with \(t_{S} \leq t\), \(\mathbf{n} = (n_{0}, \ldots, n_{S})^{T}\) be the corresponding observed population sizes, and \(\boldsymbol{\tau} = (\tau_{1}, \ldots, \tau_{S})^{T} = (t_{1} - t_{0}, \ldots, t_{S} - t_{S - 1})^{T}\) be the vector of inter-arrival times.
When the process is observed continuously the log-likelihood function is \citep[Equation (25)]{darwin1956b_behaviour}
\begin{equation}\label{eq:loglik_continuous}
  \log \mathcal{L}(\lambda, \mu | \mathbf{t}, \mathbf{n}) = B_{t} \log \lambda + D_{t} \log \mu - (\lambda + \mu) X_{t} + \sum_{s = 0}^{S - 1} \log n_{s}
\end{equation}
where \(B_{t}\) and \(D_{t}\) are respectively the total number of births and deaths recorded during the time interval \([0, t]\) while \(X_{t} = \sum_{s = 0}^{S} n_{s} \tau_{s + 1}\) is the total time lived in the population during \([0, t]\).
By convention we set \(\tau_{S + 1} = t - t_{S}\).
From (\ref{eq:loglik_continuous}) we obtain the maximum likelihood estimators (MLEs) of \(\lambda\) and \(\mu\) as
\begin{equation}\label{eq:mle_continuous}
  \hat{\lambda} = \frac{B_{t}}{X_{t}}, \qquad \hat{\mu} = \frac{D_{t}}{X_{t}}
\end{equation}
from which follows that the MLE of the growth rate \(\theta = \lambda - \mu\) is \(\hat{\theta} = \hat{\lambda} - \hat{\mu} = (B_{t} - D_{t}) / X_{t}\).
A more challenging situation is encountered when the BDP is observed discretely at fixed time points.
Rewrite the probability of transitioning from \(i\) to \(j\) in \(t\) time with birth rate \(\lambda\) and death rate \(\mu\) as \(p(j | i, t, \lambda, \mu)\).
Since the BDP is a continuous time Markov chain \citep{kendall1949jrsssbsm_stochastic} we can write the likelihood function as
\begin{equation*}
  \mathcal{L}(\lambda, \mu | \mathbf{t}, \mathbf{n}) = \prod_{s = 1}^{S} p(n_{s} | n_{s - 1}, \tau_{s}, \lambda, \mu)
\end{equation*}
Note that the joint likelihood of \(M\) observations of stochastically independent processes, having the same birth and death rates, is simply the product of the \(M\) likelihoods associated with each process.
To the best of our knowledge, no known closed form solutions for \(\hat{\lambda}\) and \(\hat{\mu}\) are currently available.
However, in the case of equidistant sampling where \(\tau_{s} = \tau\) for all \(s\), we know that \citep{keiding1975as_maximum}
\begin{equation}\label{eq:mle_theta}
  \hat{\theta} = \frac{1}{\tau}\log\left(\frac{n_{1} + \cdots + n_{S}}{n_{0} + \cdots + n_{S - 1}}\right)
\end{equation}
It is easy to show that the first moment of \(\hat{\theta}\) does not exist.
Starting with \(S = 1\) we have
\begin{equation*}
  \mathbb{E}[\hat{\theta}] = \frac{1}{\tau} \sum_{j = 0}^{\infty} \log\left(\frac{j}{n_{0}}\right) p(j | n_{0}, t, \lambda, \mu)
\end{equation*}
The first term in the summation is not defined because the probability of extinction is strictly positive, unless the process is a pure birth process (see equations (\ref{eq:transprob_9})-(\ref{eq:transprob_12})).
For a simple birth-and-death process without migration the population stays extinct once its size reaches a value of zero, therefore the previous result can be extended to any value \(S > 1\).
To estimate \(\hat{\lambda}\), \(\hat{\mu}\), and \(\hat{\theta}\) we must consider only observations in which the population is not immediately extinct at time point \(s = 1\).

To find the maximum likelihood estimators we will use a numerical approach, that is the Newton–Raphson method \citep[Chapter 4]{bonnans2006_numerical} applied to the log-likelihood function.
To proceed we need its gradient and Hessian matrix, that are
\begin{equation}\label{eq:gradient}
  \nabla l(\lambda, \mu | \mathbf{t}, \mathbf{n}) = \nabla \log \mathcal{L}(\lambda, \mu | \mathbf{t}, \mathbf{n}) = \sum_{s = 1}^{S}
  \begin{pmatrix}
    \dfrac{\partial}{\partial \lambda} \log p(n_{s} | n_{s - 1}, \tau_{s}, \lambda, \mu)\\[1em]
    \dfrac{\partial}{\partial \mu} \log p(n_{s} | n_{s - 1}, \tau_{s}, \lambda, \mu)
  \end{pmatrix}
\end{equation}
\begin{equation}\label{eq:hessian}
  \mathbf{H}(\lambda, \mu | \mathbf{t}, \mathbf{n}) = \sum_{s = 1}^{S}
  \begin{pmatrix}
    \dfrac{\partial^{2}}{\partial \lambda^{2}} \log p(n_{s} | n_{s - 1}, \tau_{s}, \lambda, \mu) & \dfrac{\partial^{2}}{\partial \lambda \partial \mu} \log p(n_{s} | n_{s - 1}, \tau_{s}, \lambda, \mu)\\[1em]
    \dfrac{\partial^{2}}{\partial \mu \partial \lambda} \log p(n_{s} | n_{s - 1}, \tau_{s}, \lambda, \mu) & \dfrac{\partial^{2}}{\partial \mu^{2}} \log p(n_{s} | n_{s - 1}, \tau_{s}, \lambda, \mu)
  \end{pmatrix}
\end{equation}
with closed form solutions of partial derivatives of log-probabilities appearing in (\ref{eq:gradient}) and (\ref{eq:hessian}) given in Appendix \ref{appendix:hessian}.
They can be evaluated with our proposed TTRR approach.
Note that (\ref{eq:gradient}) and (\ref{eq:hessian}) are sums of piecewise functions with sub-domains inherited from equations (\ref{eq:transprob_2})-(\ref{eq:transprob_8}).

\section{Applications}\label{sec:applications}
All results presented so far are implemented in a free Julia \citep{bezanson2017_sr_julia} package called ``\mbox{SimpleBirthDeathProcess}''.
The package is released under a MIT software license and can be downloaded from \url{https://github.com/albertopessia/SimpleBirthDeathProcess.jl}.

Returning to the example shown in Figure \ref{fig:relative_error}, we can see from Figure \ref{fig:relative_error_stable} that our method improves significantly the accuracy of the computations.
Interestingly, although not entirely unexpected, the algorithm has a higher numerical error in the neighbourhood of the special point \(\mu = \lambda\), that is the removable singularity of equation (\ref{eq:transprob_2}).
Note that relative errors for this particular example are always less than \(10^{-10}\) and small enough for any practical application.
In Figure \ref{fig:relative_error_contour} we can see a more general example where points near the line \(\mu = \lambda\) are again associated with higher relative errors.
Also in this case they are very small and always less than \(10^{-13}\).
\begin{figure}[p]
  \centering
  \includegraphics[width=\textwidth]{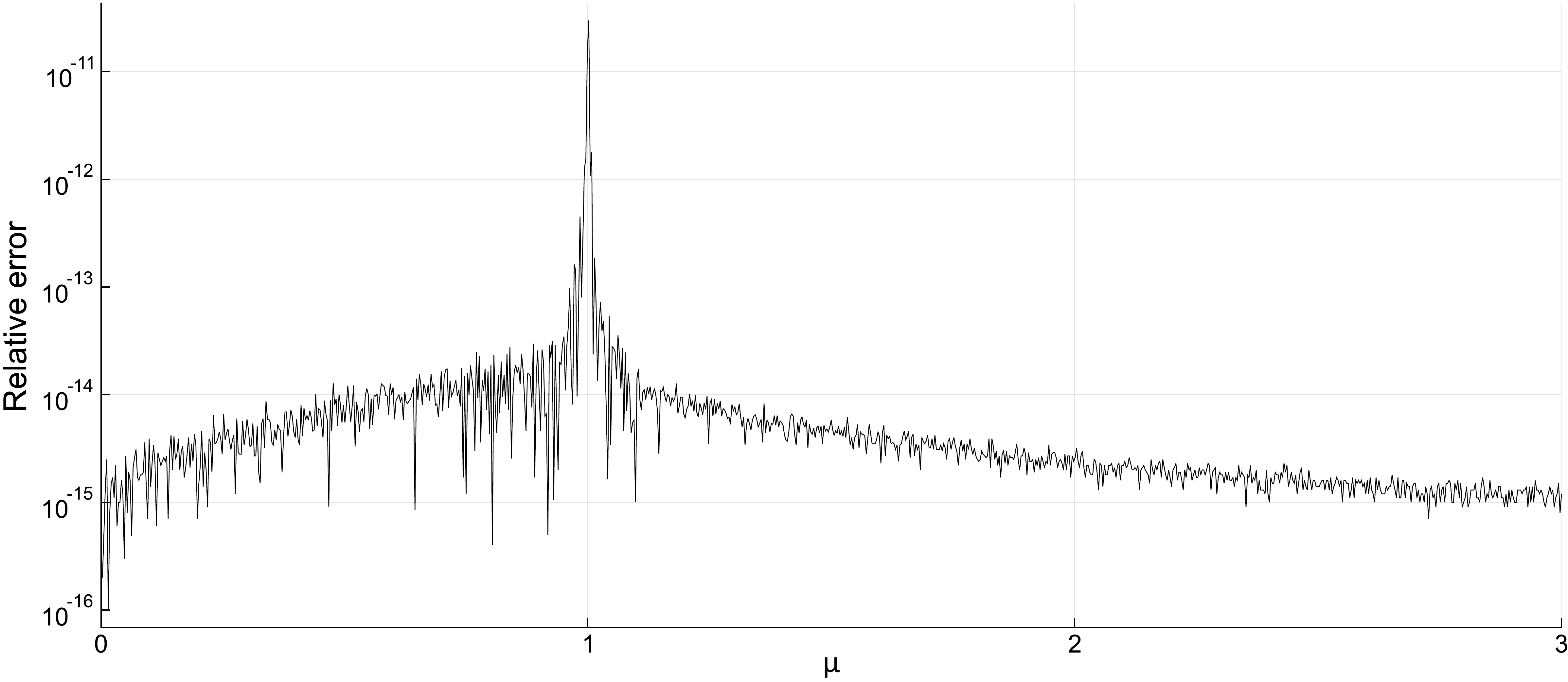}
  \caption{Numerical relative error of the log-probability evaluated using the hypergeometric representation and the TTRR approach. Parameters are the same as in Figure \ref{fig:relative_error}, that is \(i = 25\), \(j = 35\), \(t = 2\), and \(\lambda = 1\). Relative error is always less than \(10^{-10}\).}\label{fig:relative_error_stable}
\end{figure}
\begin{figure}[p]
  \centering
  \includegraphics[width=0.7\textwidth]{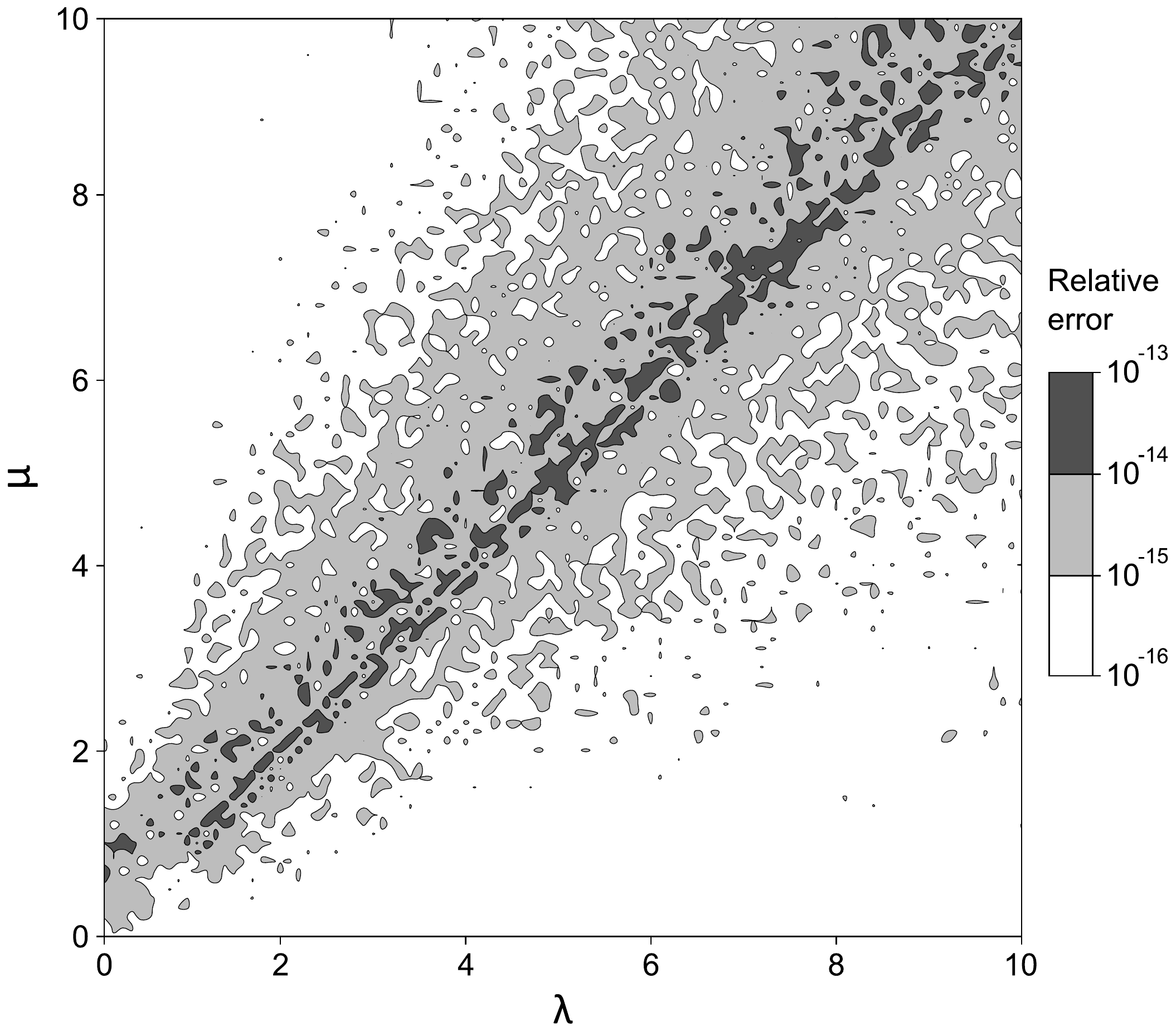}
  \caption{Numerical relative error of the log-probability evaluated using the hypergeometric representation and the TTRR approach. Parameters for this example are \(i = 200\), \(j = 100\), and \(t = 1\). Relative error is always less than \(10^{-13}\).}\label{fig:relative_error_contour}
\end{figure}

\subsection{Simulated data}
We will now study some properties of the maximum likelihood estimator of the birth rate \(\lambda\), death rate \(\mu\), and growth rate \(\theta = \lambda - \mu\).
We will use our software package to perform simulations and apply standard Monte Carlo integration to approximate the bias and root mean square error (RMSE) of MLEs.
The total number of simulations is set to \(10^{5}\) in each of the following synthetic experiments.

\subsubsection*{Constant growth rate}
The first study mimics a situation in which both rate parameters are strictly positive.
For simplicity we fix the total observation time to \(t = 10\) and assume the process to be observed at \(S\) equidistant time points, that is every \(\tau = t / S\) amount of time.
To reduce the amount of possible combinations to test we choose birth and death rates so that the expected population size and standard deviation at time \(t\) is approximately proportional to the initial population size.
In what follows we will always condition our estimators only to populations that are not immediately extinct, as explained in Section \ref{sec:likelihood}.

Results of the simulations when \(\lambda > \mu\) are shown in Table \ref{tab:monte_carlo_birth_greater_death} while results of the simulations when \(\lambda < \mu\) are shown in Table \ref{tab:monte_carlo_birth_less_death}.
\begin{table}[t]
  \centering\small
  \caption{\small Monte Carlo estimates from \(10^{5}\) simulations of a simple BDP where \(\lambda > \mu\). Growth rate \(\theta = \lambda - \mu = 0.0693\) for each row, i.e. the expected population size at time \(t = 10\) is set to be two times the initial population size \(n_{0}\). For each number of time points \(S\), the three rows correspond respectively to a standard deviation of 1.25, 1.5, and 2.0 times the initial population size \(n_{0}\).}
  \label{tab:monte_carlo_birth_greater_death}
  \begin{tabular}{@{} rr rrr rrr rr @{}}
    \toprule
    \multicolumn{2}{c}{} & \multicolumn{3}{c}{\(\lambda\)} & \multicolumn{3}{c}{\(\mu\)} & \multicolumn{2}{c}{\(\theta\)}\\
    \cmidrule{3-5} \cmidrule(l){6-8} \cmidrule(l){9-10}
    \(n_{0}\) & \(S\) & Truth & Bias & RMSE & Truth & Bias & RMSE & Bias & RMSE\\
    \midrule
        10 & 1 &  0.305 &  -0.244 &  0.249 &  0.236 &  -0.222 &  0.226 & -0.022 & 0.078\\
           &   &  0.425 &  -0.362 &  0.366 &  0.355 &  -0.335 &  0.339 & -0.027 & 0.092\\
           &   &  0.728 &  -0.659 &  0.662 &  0.658 &  -0.634 &  0.636 & -0.025 & 0.105\\
    \cmidrule{2-10}
           & 8 &  0.305 &  -0.045 &  0.141 &  0.236 &  -0.019 &  0.138 & -0.026 & 0.086\\
           &   &  0.425 &  -0.069 &  0.205 &  0.355 &  -0.028 &  0.203 & -0.041 & 0.123\\
           &   &  0.728 &  -0.132 &  0.391 &  0.658 &  -0.042 &  0.390 & -0.090 & 0.226\\
    \midrule
       100 & 1 &  2.742 &  -2.681 &  2.681 &  2.673 &  -2.658 &  2.658 & -0.023 & 0.083\\
           &   &  3.934 &  -3.872 &  3.872 &  3.864 &  -3.841 &  3.842 & -0.030 & 0.101\\
           &   &  6.966 &  -6.898 &  6.898 &  6.897 &  -6.867 &  6.867 & -0.031 & 0.118\\
    \cmidrule{2-10}
           & 8 &  2.742 &  -0.382 &  1.292 &  2.673 &  -0.357 &  1.285 & -0.026 & 0.086\\
           &   &  3.934 &  -0.565 &  1.853 &  3.864 &  -0.524 &  1.841 & -0.041 & 0.121\\
           &   &  6.966 &  -1.018 &  3.430 &  6.897 &  -0.928 &  3.413 & -0.089 & 0.226\\
    \midrule
      1000 & 1 & 27.111 & -27.050 & 27.050 & 27.041 & -27.026 & 27.026 & -0.024 & 0.085\\
           &   & 39.024 & -38.962 & 38.962 & 38.955 & -38.932 & 38.932 & -0.031 & 0.102\\
           &   & 69.349 & -69.281 & 69.281 & 69.280 & -69.249 & 69.249 & -0.032 & 0.121\\
    \cmidrule{2-10}
           & 8 & 27.111 &  -3.700 & 12.914 & 27.041 &  -3.675 & 12.906 & -0.025 & 0.085\\
           &   & 39.024 &  -5.522 & 18.489 & 38.955 &  -5.481 & 18.478 & -0.041 & 0.122\\
           &   & 69.349 &  -9.961 & 33.635 & 69.280 &  -9.871 & 33.616 & -0.090 & 0.226\\
    \bottomrule
  \end{tabular}
\end{table}
\begin{table}[t]
  \centering\small
  \caption{\small Monte Carlo estimates from \(10^{5}\) simulations of a simple BDP where \(\lambda < \mu\). Growth rate \(\theta = \lambda - \mu = -0.0693\) for each row, i.e. the expected population size at time \(t = 10\) is set to be half the initial population size \(n_{0}\). For each number of time points \(S\), the three rows correspond respectively to a standard deviation of 0.25, 0.5, and 1.0 times the initial population size \(n_{0}\).}
  \label{tab:monte_carlo_birth_less_death}
  \begin{tabular}{@{} rr rrr rrr rr @{}}
    \toprule
    \multicolumn{2}{c}{} & \multicolumn{3}{c}{\(\lambda\)} & \multicolumn{3}{c}{\(\mu\)} & \multicolumn{2}{c}{\(\theta\)}\\
    \cmidrule{3-5} \cmidrule(l){6-8} \cmidrule(l){9-10}
    \(n_{0}\) & \(S\) & Truth & Bias & RMSE & Truth & Bias & RMSE & Bias & RMSE\\
    \midrule
        10 & 1 &   0.052 &   -0.052 &   0.052 &   0.121 &   -0.039 &   0.067 & -0.012 & 0.057\\
           &   &   0.312 &   -0.306 &   0.306 &   0.381 &   -0.295 &   0.304 & -0.011 & 0.085\\
           &   &   1.352 &   -1.318 &   1.319 &   1.421 &   -1.370 &   1.372 &  0.051 & 0.116\\
    \cmidrule{2-10}
           & 8 &   0.052 &   -0.007 &   0.046 &   0.121 &    0.005 &   0.062 & -0.012 & 0.057\\
           &   &   0.312 &   -0.053 &   0.208 &   0.381 &    0.000 &   0.225 & -0.053 & 0.152\\
           &   &   1.352 &   -0.279 &   0.970 &   1.421 &   -0.035 &   0.905 & -0.245 & 0.476\\
    \midrule
       100 & 1 &   0.832 &   -0.831 &   0.831 &   0.901 &   -0.816 &   0.819 & -0.015 & 0.061\\
           &   &   3.431 &   -3.425 &   3.425 &   3.500 &   -3.396 &   3.398 & -0.029 & 0.112\\
           &   &  13.828 &  -13.796 &  13.797 &  13.898 &  -13.835 &  13.835 &  0.038 & 0.128\\
    \cmidrule{2-10}
           & 8 &   0.832 &   -0.116 &   0.419 &   0.901 &   -0.105 &   0.416 & -0.011 & 0.054\\
           &   &   3.431 &   -0.461 &   1.702 &   3.500 &   -0.409 &   1.695 & -0.052 & 0.146\\
           &   &  13.828 &   -1.821 &   8.481 &  13.898 &   -1.556 &   8.365 & -0.265 & 0.537\\
    \midrule
      1000 & 1 &   8.630 &   -8.629 &   8.629 &   8.699 &   -8.615 &   8.615 & -0.015 & 0.061\\
           &   &  34.623 &  -34.617 &  34.617 &  34.692 &  -34.585 &  34.586 & -0.032 & 0.119\\
           &   & 138.595 & -138.563 & 138.563 & 138.664 & -138.600 & 138.600 &  0.037 & 0.132\\
    \cmidrule{2-10}
           & 8 &   8.630 &   -1.107 &   4.163 &   8.699 &   -1.096 &   4.160 & -0.011 & 0.054\\
           &   &  34.623 &   -4.462 &  16.635 &  34.692 &   -4.411 &  16.627 & -0.052 & 0.145\\
           &   & 138.595 &  -16.466 &  83.001 & 138.664 &  -16.198 &  82.886 & -0.267 & 0.545\\
    \bottomrule
  \end{tabular}
\end{table}

Estimators are generally negatively biased but we also observe situations when \(\lambda < \mu\) in which the bias is positive.
The magnitude of the bias of \(\hat{\lambda}\) and \(\hat{\mu}\) is very large when only one time point is used, for which we have \(|\text{Bias}(\hat{\lambda})| \approx \text{RMSE}(\hat{\lambda})\) and \(|\text{Bias}(\hat{\mu})| \approx \text{RMSE}(\hat{\mu})\).
Increasing the number of time points \(S\) help reducing both the bias and RMSE of \(\hat{\lambda}\) and \(\hat{\mu}\).
All estimators obviously perform worse when the standard deviation of the stochastic process is high.
What is surprising to us is the observation that \(\hat{\theta}\) has approximately the same performance regardless of the initial sample size.
Increasing the number of time points has also the counter-intuitive effect of making the estimation worse.

\subsubsection*{Technical replicates}
Following the results from the previous section we want to investigate the performance of the estimators when the stochastic process is observed more than once.
As an example, this is a standard setting in dose-response drug screening experiments where cell counts are observed after a period of incubation and (usually) 3 to 5 technical replicates are produced under the same experimental conditions.
For a fair comparison we will use the same simulation parameters from the previous simulation experiment with the only difference of now having three technical replicates instead of one.
Results of the simulations when \(\lambda > \mu\) are shown in Table \ref{tab:monte_carlo_replicates_brg} while results of the simulations when \(\lambda < \mu\) are shown in Table \ref{tab:monte_carlo_replicates_drg}.
\begin{table}[t]
  \centering\small
  \caption{\small Monte Carlo estimates from \(10^{5}\) simulations of three simple BDPs where \(\lambda > \mu\). Growth rate \(\theta = \lambda - \mu = 0.0693\) for each row, i.e. the expected population size at time \(t = 10\) is set to be two times the initial population size \(n_{0}\). For each number of time points \(S\), the three rows correspond respectively to a standard deviation of 1.25, 1.5, and 2.0 times the initial population size \(n_{0}\).}
  \label{tab:monte_carlo_replicates_brg}
  \begin{tabular}{@{} rr rrr rrr rr @{}}
    \toprule
    \multicolumn{2}{c}{} & \multicolumn{3}{c}{\(\lambda\)} & \multicolumn{3}{c}{\(\mu\)} & \multicolumn{2}{c}{\(\theta\)}\\
    \cmidrule{3-5} \cmidrule(l){6-8} \cmidrule(l){9-10}
    \(n_{0}\) & \(k\) & Truth & Bias & RMSE & Truth & Bias & RMSE & Bias & RMSE\\
    \midrule
        10 & 1 &  0.305 &  -0.107 &  0.170 &  0.236 &  -0.102 &  0.164 & -0.006 & 0.038\\
           &   &  0.425 &  -0.185 &  0.242 &  0.355 &  -0.179 &  0.234 & -0.005 & 0.045\\
           &   &  0.728 &  -0.427 &  0.469 &  0.658 &  -0.430 &  0.466 &  0.003 & 0.054\\
    \cmidrule{2-10}
           & 8 &  0.305 &  -0.014 &  0.082 &  0.236 &  -0.006 &  0.080 & -0.008 & 0.039\\
           &   &  0.425 &  -0.020 &  0.117 &  0.355 &  -0.009 &  0.116 & -0.011 & 0.049\\
           &   &  0.728 &  -0.033 &  0.209 &  0.658 &  -0.011 &  0.208 & -0.022 & 0.075\\
    \midrule
       100 & 1 &  2.742 &  -1.064 &  1.712 &  2.673 &  -1.057 &  1.707 & -0.006 & 0.039\\
           &   &  3.934 &  -1.793 &  2.390 &  3.864 &  -1.787 &  2.382 & -0.006 & 0.046\\
           &   &  6.966 &  -4.201 &  4.610 &  6.897 &  -4.202 &  4.606 &  0.001 & 0.054\\
    \cmidrule{2-10}
           & 8 &  2.742 &  -0.126 &  0.765 &  2.673 &  -0.119 &  0.763 & -0.007 & 0.039\\
           &   &  3.934 &  -0.176 &  1.099 &  3.864 &  -0.165 &  1.097 & -0.011 & 0.049\\
           &   &  6.966 &  -0.270 &  1.986 &  6.897 &  -0.248 &  1.984 & -0.021 & 0.075\\
    \midrule
      1000 & 1 & 27.111 & -10.559 & 17.127 & 27.041 & -10.553 & 17.122 & -0.006 & 0.039\\
           &   & 39.024 & -17.871 & 23.869 & 38.955 & -17.865 & 23.860 & -0.006 & 0.046\\
           &   & 69.349 & -41.816 & 45.913 & 69.280 & -41.818 & 45.909 &  0.001 & 0.055\\
    \cmidrule{2-10}
           & 8 & 27.111 &  -1.275 &  7.621 & 27.041 &  -1.268 &  7.620 & -0.007 & 0.039\\
           &   & 39.024 &  -1.673 & 10.999 & 38.955 &  -1.662 & 10.997 & -0.011 & 0.049\\
           &   & 69.349 &  -2.647 & 19.663 & 69.280 &  -2.625 & 19.661 & -0.021 & 0.075\\
    \bottomrule
  \end{tabular}
\end{table}
\begin{table}[t]
  \centering\small
  \caption{\small \small Monte Carlo estimates from \(10^{5}\) simulations of three simple BDPs where \(\lambda < \mu\). Growth rate \(\theta = \lambda - \mu = -0.0693\) for each row, i.e. the expected population size at time \(t = 10\) is set to be half the initial population size \(n_{0}\). For each number of time points \(S\), the three rows correspond respectively to a standard deviation of 0.25, 0.5, and 1.0 times the initial population size \(n_{0}\).}
  \label{tab:monte_carlo_replicates_drg}
  \begin{tabular}{@{} rr rrr rrr rr @{}}
    \toprule
    \multicolumn{2}{c}{} & \multicolumn{3}{c}{\(\lambda\)} & \multicolumn{3}{c}{\(\mu\)} & \multicolumn{2}{c}{\(\theta\)}\\
    \cmidrule{3-5} \cmidrule(l){6-8} \cmidrule(l){9-10}
    \(n_{0}\) & \(k\) & Truth & Bias & RMSE & Truth & Bias & RMSE & Bias & RMSE\\
    \midrule
        10 & 1 &   0.052 &   -0.024 &   0.050 &   0.121 &   -0.021 &   0.052 & -0.003 & 0.030\\
           &   &   0.312 &   -0.215 &   0.237 &   0.381 &   -0.224 &   0.240 &  0.009 & 0.051\\
           &   &   1.352 &   -1.198 &   1.209 &   1.421 &   -1.269 &   1.277 &  0.071 & 0.106\\
    \cmidrule{2-10}
           & 8 &   0.052 &   -0.002 &   0.027 &   0.121 &    0.001 &   0.033 & -0.004 & 0.030\\
           &   &   0.312 &   -0.016 &   0.111 &   0.381 &   -0.001 &   0.113 & -0.015 & 0.066\\
           &   &   1.352 &   -0.090 &   0.468 &   1.421 &   -0.022 &   0.454 & -0.068 & 0.178\\
    \midrule
       100 & 1 &   0.832 &   -0.305 &   0.588 &   0.901 &   -0.301 &   0.584 & -0.004 & 0.030\\
           &   &   3.431 &   -2.115 &   2.339 &   3.500 &   -2.117 &   2.333 &  0.002 & 0.055\\
           &   &  13.828 &  -12.336 &  12.447 &  13.898 &  -12.399 &  12.507 &  0.063 & 0.108\\
    \cmidrule{2-10}
           & 8 &   0.832 &   -0.038 &   0.246 &   0.901 &   -0.035 &   0.246 & -0.004 & 0.029\\
           &   &   3.431 &   -0.136 &   0.986 &   3.500 &   -0.121 &   0.985 & -0.015 & 0.065\\
           &   &  13.828 &   -0.498 &   4.210 &  13.898 &   -0.427 &   4.196 & -0.071 & 0.185\\
    \midrule
      1000 & 1 &   8.630 &   -3.055 &   5.873 &   8.699 &   -3.050 &   5.870 & -0.004 & 0.030\\
           &   &  34.623 &  -20.741 &  22.973 &  34.692 &  -20.742 &  22.967 &  0.001 & 0.054\\
           &   & 138.595 & -115.124 & 116.248 & 138.664 & -115.200 & 116.320 &  0.076 & 0.076\\
    \cmidrule{2-10}
           & 8 &   8.630 &   -0.371 &   2.462 &   8.699 &   -0.367 &   2.461 & -0.004 & 0.029\\
           &   &  34.623 &   -1.310 &   9.754 &  34.692 &   -1.295 &   9.753 & -0.016 & 0.066\\
           &   & 138.595 &   -4.562 &  41.566 & 138.664 &   -4.490 &  41.555 & -0.071 & 0.185\\
    \bottomrule
  \end{tabular}
\end{table}

As expected, we see a decrease in both bias magnitude and RMSE for \(\hat{\lambda}\) and \(\hat{\mu}\).
A small improvement is obtained also for \(\hat{\theta}\).
Again, increasing the number of time points allow for a better estimation of \(\lambda\) and \(\mu\) but make the estimation of \(\theta\) worse.
When increasing the number of time points \(S\), the loss (gain) of performance is lower (higher) that in the single observation case of the previous section.

\subsubsection*{Real data}
As an example application we will use real data from a cancer drug combination experiment originally performed and analysed by \citet{liu2007sm_analysis}.
Briefly, two monoclonal antibodies were combined together at a concentration ratio of 1:1 to form a mixture.
Tested concentrations of the mixture were 0 (no antibody), 0.025, 0.25, 2.5, and 10 \textmugreek g/ml.
Living cell counts were subsequently measured with a fluorescence microscopy at 1, 2, and 3 days.
For each time point they performed six technical replicates for concentrations greater than zero and twelve replicates for the control dose of zero.
Since the initial number of cells was not available, they estimated it from the data to be on average approximately equal to 23.
Following previous studies \citep{crawford_2014_jasa_estimation} we will fix for each and every observation an initial cell count of 23 as if it were known in advance.
The complete dataset is visually represented in Figure \ref{fig:liu_dataset}.
\begin{figure}[t]
  \centering
  \includegraphics[width=\textwidth]{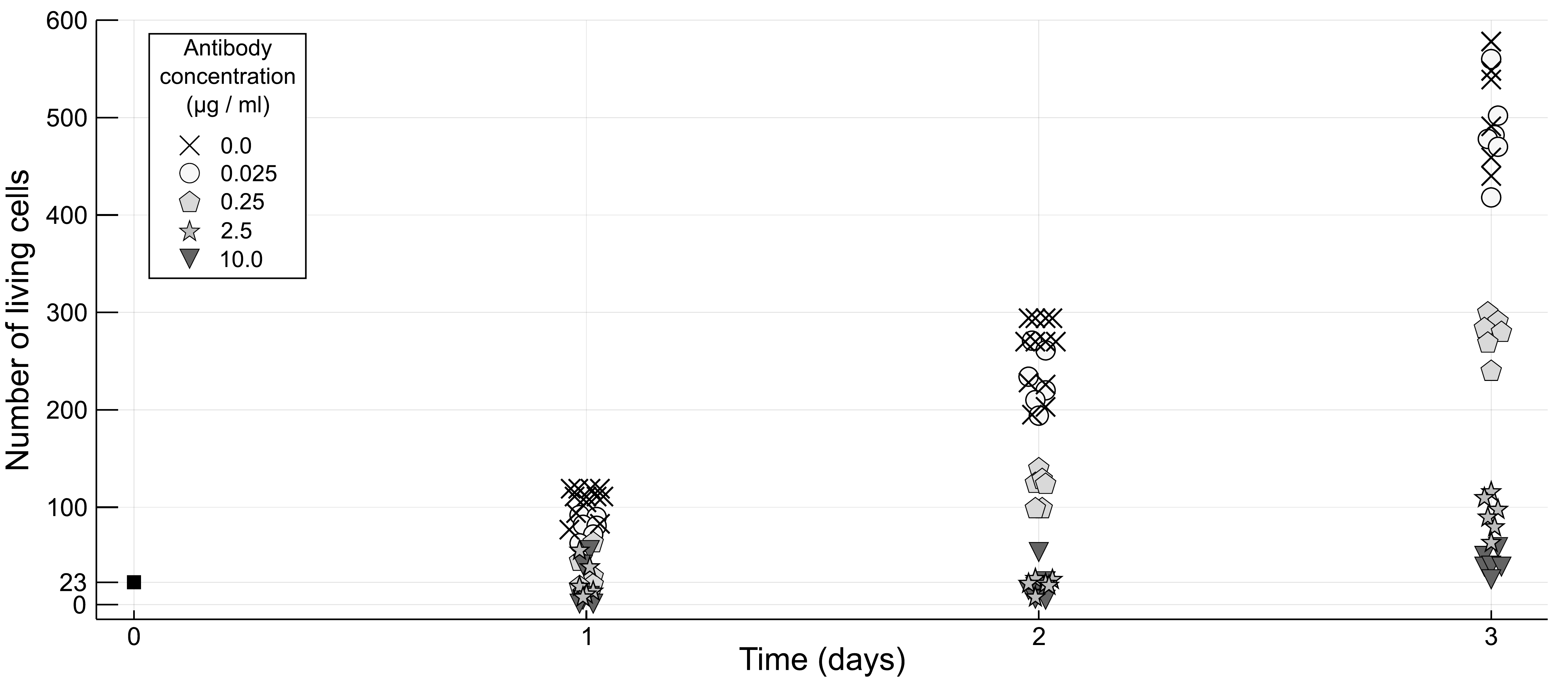}
  \caption{Antibody dataset by \citet{liu2007sm_analysis}. All the observed counts are assumed to be originated from the same number of cells \(n_{0} = 23\). Increasing the antibody concentration reduces the growth rate of cancer cells.}\label{fig:liu_dataset}
\end{figure}
It is important to note that the dataset is made of 108 independent observations, i.e. counts referring to the same concentration at different time points are not part of the same time series but are, instead, independent realizations of the same stochastic process observed at different times.
In our notation, \(S = 1\) and \(\tau = t / S = t\) for each of the 108 measurements.
The basic datum \(\mathbf{x}_{i}\), \(i = 1, \ldots, 108\), is a vector \((c_{i}, t_{i}, n_{i}(0), n_{i}(t_{i}))^{T}\) where \(c_{i}\) is the tested antibody concentration, \(t_{i}\) is the time in days, \(n_{i}(0)\) is the initial population size set to 23 for each and every observation, and \(n_{i}(t_{i})\) is the final cancer cell counts for observation \(i\).
For further details about the study and the experimental design we refer to the original article of \citet{liu2007sm_analysis}.
To model the data we use a similar approach to that of \citet{crawford_2014_jasa_estimation}, that is a linear model on the logarithm scale of the basic process rates.
Formally we define
\begin{equation}\label{eq:glm_model}
  \left\{
  \begin{aligned}
    \log(\lambda_{i}) &= \alpha_{\lambda} + \beta_{\lambda} \log(1 + c_{i})\\
    \log(\mu_{i}) &= \alpha_{\mu} + \beta_{\mu} \log(1 + c_{i})
  \end{aligned}
  \right.,\;
  \text{ for all } i = 1, \ldots, 108
\end{equation}
Maximum likelihood estimates and their corresponding standard errors (SE) are shown in Table \ref{tab:mle_glm_model}.
We obtained estimates by numerically maximizing the log-likelihood function with the BFGS algorithm \citep[Chapter 4]{bonnans2006_numerical}.
We applied the delta method to the observed Fisher information matrix in order to compute the standard error of all parameters.
\begin{table}
  \centering
  \caption{Maximum likelihood estimates of model (\ref{eq:glm_model}) based on the antibody dataset.}
  \label{tab:mle_glm_model}
  \begin{tabular}{@{} l rr rr rr @{}}
    \toprule
    & \multicolumn{2}{c}{\(\lambda\)} & \multicolumn{2}{c}{\(\mu\)} & \multicolumn{2}{c}{\(\theta = \lambda - \mu\)}\\
    \cmidrule(l){2-3} \cmidrule(l){4-5} \cmidrule(l){6-7}
     Dose (\textmugreek g/ml) & Estimate & SE & Estimate & SE & Estimate & SE\\
    \midrule
        0 & 4.0344 & 0.4844 & 2.9572 & 0.4835 &  1.0772 & 0.0292\\
    0.025 & 4.0238 & 0.4806 & 2.9595 & 0.4806 &  1.0644 & 0.0285\\
     0.25 & 3.9397 & 0.4548 & 2.9774 & 0.4595 &  0.9623 & 0.0283\\
      2.5 & 3.5304 & 0.4476 & 3.0721 & 0.4366 &  0.4583 & 0.0535\\
       10 & 3.1249 & 0.5778 & 3.1810 & 0.5962 & -0.0561 & 0.0740\\
    \bottomrule
  \end{tabular}
\end{table}

According to our model, increasing the antibody concentration has the double effect of reducing the birth rate and raising the death rate while maintaining the overall rate \(\lambda + \mu\) approximately the same.
When the dose of the treatment increases the global growth rate \(\theta\) decreases as a consequence, reaching a negative value at the maximum tested concentration.
Interestingly, \citeauthor{crawford_2014_jasa_estimation} obtained values that are slightly different from ours but still very close.
In particular, the maximum absolute difference between our estimates of \(\theta\) and theirs is just 0.054.
Since their R package \textit{birth.death} is not available for download any more we could not replicate the analysis and investigate the discrepancies more.
We believe, however, that the observed differences are simply due to numerical errors or to a chosen solution that is a local optimum rather than a global.

\section{Concluding remarks}\label{sec:conclusions}
Maximum likelihood estimators for the basic rates of a simple (linear) birth-and-process are available in closed form only when the process is observed continuously over time.
Numerical methods are currently the only option to draw inferences for discretely observed processes.
However, we showed that direct application of the well-known transition probability might be subject to large numerical error.
We rewrote the probability in terms of a Gaussian hypergeometric function and found a three-term recurrence relation for its evaluation.
Not only our approach led to very accurate approximations but also to a computational efficient algorithm when compared to the na{\"i}ve direct summation method.

By means of simulation we observed that MLEs \(\hat{\lambda}\) and \(\hat{\mu}\) are largely negatively biased.
We confirmed the intuition that to obtain better estimates it is important to employ a large initial population size, multiple time points, and multiple technical replicates.
The actual values, as one would expect, depend on the magnitude of the basic rates, i.e. the process standard deviation.
If only the growth parameter \(\theta = \lambda - \mu\) is of interest then multiple technical replicates with (surprisingly) only one time point provide the best results.
Interestingly, the initial population size seems not to affect the bias nor the root mean square error of \(\hat{\theta}\).

We also released a free Julia package called ``SimpleBirthDeathProcess''.
With the help of our tool it is possible to simulate, fit, or just evaluate the likelihood function of a simple BDP.
Accurate evaluation of the log-likelihood function will create opportunities for future research, such as implementation of MCMC algorithms for Bayesian inference.
As a final note, it might be worth investigating our conjecture that Lemma \ref{lemma:pseudominimal} can be extended to TTRRs with a positive coefficient.

\section*{Acknowledgements}\label{sec:acknowledgements}
We thank Prof Hao Liu for sharing with us the antibody dataset we analysed in this article.
We also thank Dr Gerry DeNardo and Dr Evan Tobin for having performed the experiment and collected the data.
We thank Francesco Iafrate for a valuable discussion about the proof of Theorem \ref{theorem:minimal_dominant}.

Research was supported by the European Research Council (ERC) starting grant, No 716063 (Drug-Comb), and the Academy of Finland Research Fellow grant, No 317680.

\bibliographystyle{abbrvnat}
\bibliography{lib/biblio}

\begin{thebibliography}{25}
\providecommand{\natexlab}[1]{#1}
\providecommand{\url}[1]{\texttt{#1}}
\expandafter\ifx\csname urlstyle\endcsname\relax
  \providecommand{\doi}[1]{doi: #1}\else
  \providecommand{\doi}{doi: \begingroup \urlstyle{rm}\Url}\fi

\bibitem[Arley and Borchsenius(1944)]{arley1944am_theory}
N.~Arley and V.~Borchsenius.
\newblock On the theory of infinite systems of differential equations and their
  application to the theory of stochastic processes and the perturbation theory
  of quantum mechanics.
\newblock \emph{Acta Mathematica}, 76\penalty0 (3):\penalty0 261--322, 1944.
\newblock \doi{10.1007/BF02551579}.

\bibitem[Bailey(1964)]{bailey1964_elements}
N.~T.~J. Bailey.
\newblock \emph{The Elements of Stochastic Processes with Applications to the
  Natural Sciences}.
\newblock Wiley, New York, NY, USA, 1964.
\newblock ISBN 0-471-04165-3.

\bibitem[Bezanson et~al.(2017)Bezanson, Edelman, Karpinski, and
  Shah]{bezanson2017_sr_julia}
J.~Bezanson, A.~Edelman, S.~Karpinski, and V.~B. Shah.
\newblock Julia: A fresh approach to numerical computing.
\newblock \emph{SIAM Review}, 59\penalty0 (1):\penalty0 65--98, 2017.
\newblock \doi{10.1137/141000671}.

\bibitem[Bonnans et~al.(2006)Bonnans, Gilbert, Lemarechal, and
  Sagastiz\'abal]{bonnans2006_numerical}
J.-F. Bonnans, J.~C. Gilbert, C.~Lemarechal, and C.~A. Sagastiz\'abal.
\newblock \emph{Numerical Optimization: Theoretical and Practical Aspects}.
\newblock Universitext. {Springer Berlin Heidelberg}, Heidelberg, Germany,
  second edition, 2006.
\newblock ISBN 978-3-540-35445-1.

\bibitem[Crawford and Suchard(2012)]{crawford2012jmb_transition}
F.~W. Crawford and M.~A. Suchard.
\newblock Transition probabilities for general birth-death processes with
  applications in ecology, genetics, and evolution.
\newblock \emph{Journal of Mathematical Biology}, 65\penalty0 (3):\penalty0
  553--580, 2012.
\newblock \doi{10.1007/s00285-011-0471-z}.

\bibitem[Crawford et~al.(2014)Crawford, Minin, and
  Suchard]{crawford_2014_jasa_estimation}
F.~W. Crawford, V.~N. Minin, and M.~A. Suchard.
\newblock Estimation for general birth-death processes.
\newblock \emph{Journal of the American Statistical Association}, 109\penalty0
  (506):\penalty0 730--747, 2014.
\newblock \doi{10.1080/01621459.2013.866565}.

\bibitem[Darwin(1956)]{darwin1956b_behaviour}
J.~H. Darwin.
\newblock The behaviour of an estimator for a simple birth and death process.
\newblock \emph{Biometrika}, 43\penalty0 (1/2):\penalty0 23--31, 1956.
\newblock \doi{10.2307/2333575}.

\bibitem[Dea\~no and Segura(2007)]{deano2007mc_transitory}
A.~Dea\~no and J.~Segura.
\newblock Transitory minimal solutions of hypergeometric recursions and
  pseudoconvergence of associated continued fractions.
\newblock \emph{Mathematics of Computation}, 76\penalty0 (258):\penalty0
  879--901, 2007.
\newblock \doi{10.1090/S0025-5718-07-01934-5}.

\bibitem[Elaydi(2005)]{elaydi2005_introduction}
S.~Elaydi.
\newblock \emph{An Introduction to Difference Equations}.
\newblock Undergraduate texts in mathematics. Springer-Verlag New York, New
  York, NY, USA, third edition, 2005.
\newblock ISBN 978-0-387-23059-7.

\bibitem[Feller(1939)]{feller1939ab_grundlagen}
W.~Feller.
\newblock Die grundlagen der {{Volterraschen}} theorie des kampfes ums dasein
  in wahrscheinlichkeitstheoretischer behandlung.
\newblock \emph{Acta Biotheoretica}, pages 11--40, 1939.
\newblock \doi{10.1007/978-3-319-16859-3_20}.

\bibitem[Feller(1968)]{feller1968_introduction}
W.~Feller.
\newblock \emph{An Introduction to Probability Theory and Its Applications,
  {{Vol}}. {{I}}}.
\newblock Wiley Series in Probability and Mathematical Statistics. Wiley, New
  York, NY, USA, third edition, 1968.
\newblock ISBN 978-0-471-25708-0.

\bibitem[Gil et~al.(2007)Gil, Segura, and Temme]{gil2007_numerical}
A.~Gil, J.~Segura, and N.~M. Temme.
\newblock \emph{Numerical Methods for Special Functions}.
\newblock Society for Industrial and Applied Mathematics, Philadelphia, PA,
  USA, 2007.
\newblock ISBN 978-0-89871-634-4.

\bibitem[Goldberg(1991)]{goldberg1991acs_what}
D.~Goldberg.
\newblock What every computer scientist should know about floating-point
  arithmetic.
\newblock \emph{ACM Computing Surveys}, 23\penalty0 (1):\penalty0 5--48, 1991.
\newblock \doi{10.1145/103162.103163}.

\bibitem[Higham(2002)]{higham2002_accuracy}
N.~J. Higham.
\newblock \emph{Accuracy and Stability of Numerical Algorithms}, volume~80 of
  \emph{Other Titles in Applied Mathematics}.
\newblock Society for Industrial and Applied Mathematics, Philadelphia, PA,
  USA, second edition, 2002.
\newblock ISBN 978-0-89871-802-7.

\bibitem[Keiding(1975)]{keiding1975as_maximum}
N.~Keiding.
\newblock Maximum likelihood estimation in the birth-and-death process.
\newblock \emph{The Annals of Statistics}, 3\penalty0 (2):\penalty0 363--372,
  1975.
\newblock \doi{10.1214/aos/1176343062}.

\bibitem[Kendall(1948)]{kendall1948ams_generalized}
D.~G. Kendall.
\newblock On the generalized birth-and-death process.
\newblock \emph{The Annals of Mathematical Statistics}, 19\penalty0
  (1):\penalty0 1--15, 1948.
\newblock \doi{10.1214/aoms/1177730285}.

\bibitem[Kendall(1949)]{kendall1949jrsssbsm_stochastic}
D.~G. Kendall.
\newblock Stochastic processes and population growth.
\newblock \emph{Journal of the Royal Statistical Society: Series B (Statistical
  Methodology)}, 11\penalty0 (2):\penalty0 230--282, 1949.

\bibitem[Liu et~al.(2007)Liu, Beckett, and DeNardo]{liu2007sm_analysis}
H.~Liu, L.~A. Beckett, and G.~L. DeNardo.
\newblock On the analysis of count data of birth-and-death process type: With
  application to molecularly targeted cancer therapy.
\newblock \emph{Statistics in Medicine}, 26\penalty0 (5):\penalty0 1114--1135,
  2007.
\newblock \doi{10.1002/sim.2594}.

\bibitem[Monagan et~al.(2005)Monagan, Geddes, Heal, Labahn, Vorkoetter,
  McCarron, and DeMarco]{monagan2005_maple}
M.~B. Monagan, K.~O. Geddes, K.~M. Heal, G.~Labahn, S.~M. Vorkoetter,
  J.~McCarron, and P.~DeMarco.
\newblock \emph{Maple~10 Programming Guide}.
\newblock Maplesoft, Waterloo, ON, Canada, 2005.

\bibitem[Murphy and O'Donohoe(1975)]{murphy1975ijam_properties}
J.~A. Murphy and M.~R. O'Donohoe.
\newblock Some properties of continued fractions with applications in
  {{Markov}} processes.
\newblock \emph{IMA Journal of Applied Mathematics}, 16\penalty0 (1):\penalty0
  57--71, 1975.
\newblock \doi{10.1093/imamat/16.1.57}.

\bibitem[Petkov{\v s}ek et~al.(1996)Petkov{\v s}ek, Wilf, and
  Zeilberger]{petkovsek1996_aequalb}
M.~Petkov{\v s}ek, H.~Wilf, and D.~Zeilberger.
\newblock \emph{{A = B}}.
\newblock {A K Peters/CRC Press}, Natick, MA, USA, 1996.
\newblock ISBN 978-1-56881-063-8.

\bibitem[Press et~al.(2007)Press, Teukolsky, Vetterling, and
  Flannery]{press2007_numerical}
W.~H. Press, S.~A. Teukolsky, W.~T. Vetterling, and B.~P. Flannery.
\newblock \emph{Numerical Recipes: The Art of Scientific Computing}.
\newblock Cambridge University Press, Cambridge, England, UK, third edition,
  2007.
\newblock ISBN 978-0-521-88068-8.

\bibitem[Slater(1966)]{slater1966_generalized}
L.~J. Slater.
\newblock \emph{Generalized Hypergeometric Functions}.
\newblock Cambridge University Press, Cambridge, England, UK, 1966.
\newblock ISBN 978-0-521-06483-5.

\bibitem[Stummel(1980)]{stummel1980fonc(na_rounding}
F.~Stummel.
\newblock Rounding error analysis of elementary numerical algorithms.
\newblock In \emph{Fundamentals of Numerical Computation (Computer-Oriented
  Numerical Analysis)}, volume~2 of \emph{Computing Supplementa}, pages
  169--195. {Springer Vienna}, Vienna, Austria, 1980.
\newblock ISBN 978-3-7091-8577-3.

\bibitem[Zeilberger(1991)]{zeilberger1991jsc_method}
D.~Zeilberger.
\newblock The method of creative telescoping.
\newblock \emph{Journal of Symbolic Computation}, 11\penalty0 (3):\penalty0
  195--204, 1991.
\newblock \doi{10.1016/S0747-7171(08)80044-2}.

\end{thebibliography}

\clearpage

\appendix

\section{Algorithm for evaluating \texorpdfstring{\({}_{2}F_{1}(-a, -b; -(a + b - k); -z)\)}{2F1(-a, -b; -(a + b - k); -z)}}\label{appendix:algorithm}
\renewcommand{\thealgorithm}{A.\arabic{algorithm}}
\begin{algorithm}
  \caption{Hypergeometric evaluation}
  \label{alg:hyper_eval}
  \begin{algorithmic}[1]
    \Input{Integers \(a \geq 0\), \(b \geq 0\), \(k \leq 1\). Real \(z > -1\).}
    \Output{\({}_{2}F_{1}(-a, -b; -(a + b - k); -z)\)}
    \Initialization{\(m \gets \min(a, b)\)\\\(M \gets \max(a, b)\)}
    \If{\(z = 0\) or \(m = 0\)}
      \State \Return 1
    \EndIf
    \State \(y \gets 1 + \dfrac{M z}{M + 1 - k}\)
    \If{m = 1}
      \State \Return y
    \EndIf
    \Statex // To avoid overflow define \(R_{b} = y_{b} / y_{b - 1}\), that is \(y_{b} = R_{b} y_{b - 1}\). Note that \(R_{1} = y_{1} / y_{0} = y_{1}\).
    \State \(R \gets y\)
    \For{\(n = 2, \ldots, m\)}
      \State \(R \gets 1 + \dfrac{z}{M + n - k} \left(M - n + 1 + \dfrac{(n - 1) (n - 1 - k)}{(M + n - k - 1) R}\right)\)
      \State \(y \gets R y\)
    \EndFor
    \State \Return \(y\)
  \end{algorithmic}
\end{algorithm}

\clearpage

\section{Gradient and Hessian of the log-transition probability}\label{appendix:hessian}
Partial derivatives of the log-transition probability are simple but cumbersome.
To simplify notation we will drop function arguments (unless required for clarity) and denote the first and second order partial derivatives of a function \(f(x, y)\) with \(f_{x}\), \(f_{y}\), \(f_{xx}\), \(f_{xy}\), \(f_{yx}\), and \(f_{yy}\).
We will also use the substitutions
\begin{align*}
  c^{(h)} = \binom{i}{h} \binom{i + j - h - 1}{i - 1}&, & \theta^{(h)} = \mu^{i - h} \lambda^{j - h} \phi(t, \lambda, \mu)^{i + j - 2h} \gamma(t, \lambda, \mu)^{h}&, & x &= e^{(\lambda - \mu) t}\\
  u = \frac{\F{-(i - 1), -(j - 1)}{-(i + j)}{-z(t, \lambda, \mu)}}{\F{-i, -j}{-(i + j - 1)}{-z(t, \lambda, \mu)}}&, & v = \frac{\F{-(i - 2), -(j - 2)}{-(i + j + 1)}{-z(t, \lambda, \mu)}}{\F{-i, -j}{-(i + j - 1)}{-z(t, \lambda, \mu)}}& & &
\end{align*}
Partial derivatives of the log-transition probability, in their most general form, are simply
\begin{align*}
  (\log p)_{\lambda} &= \dfrac{\sum_{h} c^{(h)} \theta_{\lambda}^{(h)}}{\sum_{k} c^{(k)} \theta^{(k)}} & (\log p)_{\mu} &= \dfrac{\sum_{h} c^{(h)} \theta_{\mu}^{(h)}}{\sum_{k} c^{(k)} \theta^{(k)}}\\
  (\log p)_{\lambda \lambda} &= \dfrac{\sum_{h} c^{(h)} \theta_{\lambda \lambda}^{(h)}}{\sum_{k} c^{(k)} \theta^{(k)}} - (\log p)_{\lambda}^{2} & (\log p)_{\lambda \mu} &= \dfrac{\sum_{h} c^{(h)} \theta_{\lambda \mu}^{(h)}}{\sum_{k} c^{(k)} \theta^{(k)}} - (\log p)_{\lambda} (\log p)_{\mu}\\
  (\log p)_{\mu \lambda} &= \dfrac{\sum_{h} c^{(h)} \theta_{\mu \lambda}^{(h)}}{\sum_{k} c^{(k)} \theta^{(k)}} - (\log p)_{\mu} (\log p)_{\lambda} & (\log p)_{\mu \mu} &= \dfrac{\sum_{h} c^{(h)} \theta_{\mu \mu}^{(h)}}{\sum_{k} c^{(k)} \theta^{(k)}} - (\log p)_{\mu}^{2}
\end{align*}
We will now list all partial derivatives of basic functions to be used later in the Section.
Partial derivatives of function \(\log\phi(t, \lambda, \mu)\) are
\begin{align*}
  (\log \phi)_{\lambda} &= - \frac{x}{\lambda x - \mu} \left(1 - \frac{(\lambda - \mu) t}{x - 1}\right), \qquad \qquad (\log \phi)_{\mu} = \frac{1}{\lambda x - \mu} \left(1 - \frac{(\lambda - \mu) t x}{x - 1}\right)\\
  (\log \phi)_{\lambda \lambda} &= \frac{(x + \mu t) x}{(\lambda x - \mu)^{2}} \left(1 - \frac{(\lambda - \mu) t}{x - 1}\right) + \frac{t x}{(\lambda x - \mu) (x - 1)} \left(1 - \frac{(\lambda - \mu) t x}{x - 1}\right)\\
  (\log \phi)_{\mu \mu} &= \frac{1 + \lambda t x}{(\lambda x - \mu)^{2}} \left(1 - \frac{(\lambda - \mu) t x}{x - 1}\right) + \frac{t x}{(\lambda x - \mu) (x - 1)} \left(1 - \frac{(\lambda - \mu) t}{x - 1}\right)\\
  (\log \phi)_{\lambda \mu} &= - \frac{(1 + \mu t) x}{(\lambda x - \mu)^{2}} \left(1 - \frac{(\lambda - \mu) t}{x - 1}\right) - \frac{t x}{(\lambda x - \mu) (x - 1)} \left(1 - \frac{(\lambda - \mu) t x}{x - 1}\right)\\
  (\log \phi)_{\mu \lambda} &= (\log \phi)_{\lambda \mu}
\end{align*}
Partial derivatives of function \(z(t, \lambda, \mu)\) are
\begin{align*}
  z_{\lambda} &= \frac{(\lambda - \mu) x}{\lambda \mu (x - 1)^{2}} \left(\frac{\lambda + \mu}{\lambda} - \frac{(\lambda - \mu) t (x + 1)}{x - 1} \right)\\
  z_{\mu} &= -\frac{(\lambda - \mu) x}{\lambda \mu (x - 1)^{2}} \left(\frac{\lambda + \mu}{\mu} - \frac{(\lambda - \mu) t (x + 1)}{x - 1} \right)\\
  z_{\lambda \lambda} &= \frac{x}{\lambda \mu (x - 1)^{2}} \left(2 \left(\frac{\mu}{\lambda}\right)^{2} - \frac{(\lambda - \mu) t}{x - 1} \left(2 \left(\frac{\lambda + \mu}{\lambda}\right) (x + 1) - \frac{(\lambda - \mu) t}{x - 1} (x^{2} + 4 x + 1)\right) \right)\\
  z_{\mu \mu} &= \frac{x}{\lambda \mu (x - 1)^{2}} \left(2 \left(\frac{\lambda}{\mu}\right)^{2} - \frac{(\lambda - \mu) t}{x - 1} \left(2 \left(\frac{\lambda + \mu}{\mu}\right) (x + 1) - \frac{(\lambda - \mu) t}{x - 1} (x^{2} + 4 x + 1)\right) \right)\\
  z_{\lambda \mu} &= -\frac{x}{\lambda \mu (x - 1)^{2}} \left(\frac{\lambda^2 + \mu^2}{\lambda \mu} - \frac{(\lambda - \mu) t}{x - 1} \left(\frac{(\lambda + \mu)^{2}}{\lambda \mu} (x + 1) - \frac{(\lambda - \mu) t}{x - 1} (x^{2} + 4 x + 1)\right) \right)\\
  z_{\mu \lambda} &= z_{\lambda \mu}
\end{align*}
Partial derivatives of function \(\log({}_{2}F_{1}(-i, -j; -(i + j - 1); -z(t, \lambda, \mu)))\) are
\begin{align*}
  (\log {}_{2}F_{1})_{\lambda} &= \frac{i j u}{i + j - 1} z_{\lambda}, \qquad \qquad (\log {}_{2}F_{1})_{\mu} = \frac{i j u}{i + j - 1} z_{\mu}\\
  (\log {}_{2}F_{1})_{\lambda \lambda} &= \frac{i j u}{i + j - 1} \left(z_{\lambda \lambda} + z_{\lambda}^{2} \left(\frac{(i - 1) (j - 1)}{i + j - 2} \frac{v}{u} - \frac{i j u}{i + j - 1}\right)\right)\\
  (\log {}_{2}F_{1})_{\mu \mu} &= \frac{i j u}{i + j - 1} \left(z_{\mu \mu} + z_{\mu}^{2} \left(\frac{(i - 1) (j - 1)}{i + j - 2} \frac{v}{u} - \frac{i j u}{i + j - 1}\right)\right)\\
  (\log {}_{2}F_{1})_{\lambda \mu} &= \frac{i j u}{i + j - 1} \left(z_{\lambda \mu} + z_{\lambda} z_{\mu} \left(\frac{(i - 1) (j - 1)}{i + j - 2} \frac{v}{u} - \frac{i j u}{i + j - 1}\right)\right)\\
  (\log {}_{2}F_{1})_{\mu \lambda} &= (\log {}_{2}F_{1})_{\lambda \mu}
\end{align*}
We can now study the shape of the partial derivatives of the log-transition probability in the various sub-domains.
Considering that the binomial coefficient \(\binom{a}{b}\) is equal to zero for all \(b > a\), we will use the convention that \(\binom{a}{b} / \binom{a}{b}\) is always equal to 1 for all \(a\) and \(b\).

\subsection*{Parameters greater than zero}
When \(t\), \(\lambda\), and \(\mu\) are all greater than zero we can safely use representation (\ref{eq:transprob_hyper}).
We need to distinguish the case \(\mu \neq \lambda\) from the case \(\mu = \lambda\).

\subsubsection*{Unequal rates}
If \(\mu \neq \lambda\) the partial derivatives are simply
\begin{align}
  (\log p)_{\lambda} &= \frac{j}{\lambda} + (i + j) (\log \phi)_{\lambda} + (\log {}_{2}F_{1})_{\lambda} \label{eq:gradient_lambda_unequal}\\
  (\log p)_{\mu} &= \frac{i}{\mu} + (i + j) (\log \phi)_{\mu} + (\log {}_{2}F_{1})_{\mu}\label{eq:gradient_mu_unequal}\\
  (\log p)_{\lambda \lambda} &= - \frac{j}{\lambda^{2}} + (i + j) (\log \phi)_{\lambda \lambda} + (\log {}_{2}F_{1})_{\lambda \lambda}\label{eq:hessian_lambda_lambda_unequal}\\
  (\log p)_{\mu \mu} &= - \frac{i}{\mu^{2}} + (i + j) (\log \phi)_{\mu \mu} + (\log {}_{2}F_{1})_{\mu \mu}\label{eq:hessian_mu_mu_unequal}\\
  (\log p)_{\lambda \mu} &= (i + j) (\log \phi)_{\lambda \mu} + (\log {}_{2}F_{1})_{\lambda \mu}\label{eq:hessian_lambda_mu_unequal}\\
  (\log p)_{\mu \lambda} &= (\log p)_{\lambda \mu}\label{eq:hessian_mu_lambda_unequal}
\end{align}

\subsubsection*{Equal rates}
Apply the limit \(\mu \rightarrow \lambda\) directly to equations (\ref{eq:gradient_lambda_unequal})-(\ref{eq:hessian_mu_lambda_unequal}) to get
\begin{align}
  \left. (\log p)_{\lambda} \right|_{\mu = \lambda} &= \frac{j}{\lambda} + (i + j) \left. (\log \phi)_{\lambda} \right|_{\mu = \lambda} + \left. (\log {}_{2}F_{1})_{\lambda} \right|_{\mu = \lambda} \label{eq:gradient_lambda_equal}\\
  \left. (\log p)_{\mu} \right|_{\mu = \lambda} &= \frac{i}{\lambda} + (i + j) \left. (\log \phi)_{\mu} \right|_{\mu = \lambda} + \left. (\log {}_{2}F_{1})_{\mu} \right|_{\mu = \lambda}\label{eq:gradient_mu_equal}\\
  \left. (\log p)_{\lambda \lambda} \right|_{\mu = \lambda} &= - \frac{j}{\lambda^{2}} + (i + j) \left. (\log \phi)_{\lambda \lambda} \right|_{\mu = \lambda} + \left. (\log {}_{2}F_{1})_{\lambda \lambda} \right|_{\mu = \lambda}\label{eq:hessian_lambda_lambda_equal}\\
  \left. (\log p)_{\mu \mu} \right|_{\mu = \lambda} &= - \frac{i}{\lambda^{2}} + (i + j) \left. (\log \phi)_{\mu \mu} \right|_{\mu = \lambda} + \left. (\log {}_{2}F_{1})_{\mu \mu} \right|_{\mu = \lambda}\label{eq:hessian_mu_mu_equal}\\
  \left. (\log p)_{\lambda \mu} \right|_{\mu = \lambda} &= (i + j) \left. (\log \phi)_{\lambda \mu} \right|_{\mu = \lambda} + \left. (\log {}_{2}F_{1})_{\lambda \mu} \right|_{\mu = \lambda}\label{eq:hessian_lambda_mu_equal}\\
  \left. (\log p)_{\mu \lambda} \right|_{\mu = \lambda} &= \left. (\log p)_{\lambda \mu} \right|_{\mu = \lambda}\label{eq:hessian_mu_lambda_equal}
\end{align}
where
\begin{align*}
  \left. (\log \phi)_{\lambda} \right|_{\mu = \lambda} &= \left. (\log \phi)_{\mu} \right|_{\mu = \lambda} = - \frac{t}{2 (1 + \lambda t)}\\
  \left. (\log \phi)_{\lambda \lambda} \right|_{\mu = \lambda} &= \left. (\log \phi)_{\mu \mu} \right|_{\mu = \lambda} = \frac{(1 - 2 \lambda t) t^2}{12 (1 + \lambda t)^{2}}\\
  \left. (\log \phi)_{\lambda \mu} \right|_{\mu = \lambda} &= \left. (\log \phi)_{\mu \lambda} \right|_{\mu = \lambda} = \frac{(5 + 2 \lambda t) t^{2}}{12 (1 + \lambda t)^{2}}\\
  \left. (\log {}_{2}F_{1})_{\lambda} \right|_{\mu = \lambda} &= \left. (\log {}_{2}F_{1})_{\mu} \right|_{\mu = \lambda} = - \frac{i j u}{(i + j - 1) \lambda^{3} t^{2}}\\
  \left. (\log {}_{2}F_{1})_{\lambda \lambda} \right|_{\mu = \lambda} &= \frac{i j}{(i + j - 1) \lambda^{4} t^{2}} \left(\frac{(12 - \lambda^{2} t^{2}) u}{6} + \frac{(i - 1) (j - 1) v}{(i + j - 2) \lambda^{2} t^{2}} - \frac{i j u^{2}}{(i + j - 1) \lambda^{2} t^{2}}\right)\\
  \left. (\log {}_{2}F_{1})_{\mu \mu} \right|_{\mu = \lambda} &= \left. (\log {}_{2}F_{1})_{\lambda \lambda} \right|_{\mu = \lambda}\\
  \left. (\log {}_{2}F_{1})_{\lambda \mu} \right|_{\mu = \lambda} &= \frac{i j}{(i + j - 1) \lambda^{4} t^{2}} \left(\frac{(6 + \lambda^{2} t^{2}) u}{6} + \frac{(i - 1) (j - 1) v}{(i + j - 2) \lambda^{2} t^{2}} - \frac{i j u^{2}}{(i + j - 1) \lambda^{2} t^{2}}\right)\\
  \left. (\log {}_{2}F_{1})_{\mu \lambda} \right|_{\mu = \lambda} &= \left. (\log {}_{2}F_{1})_{\lambda \mu} \right|_{\mu = \lambda}
\end{align*}
Note that functions \(u\) and \(v\) must be evaluated at the point \(z(t, \lambda, \lambda) = (\lambda t)^{-2} - 1\).

\subsection*{Parameters equal to zero}
When any of \(t\), \(\lambda\), or \(\mu\) equal zero it is easier to compute the partial derivatives starting from the standard representation (\ref{eq:transprob_2}) instead of (\ref{eq:transprob_hyper}).
However, derivatives of \(\theta^{(h)}\) are long and complicated, especially the second-order partial derivatives.
Considering that intermediate results are not of interest we won't write them here.
Instead, we will only provide the required final solutions.

\subsubsection*{Observation time is zero}
When \(t = 0\) the partial derivatives are always zero regardless of the values of \(i, j, \lambda\), or \(\mu\).
This is a consequence of the fact that the transition probability, equations (\ref{eq:transprob_6}) and (\ref{eq:transprob_8}), does not depend on the process rates.

\subsubsection*{Death rate is zero}
When \(\mu\) approaches zero also the partial derivatives of \(\theta^{(h)}\), in general, approach zero.
Only exceptions are \(\theta_{\lambda}^{(i)}\), \(\theta_{\mu}^{(i)}\), \(\theta_{\mu}^{(i - 1)}\), \(\theta_{\lambda \lambda}^{(i)}\), \(\theta_{\mu \mu}^{(i)}\),\(\theta_{\mu \mu}^{(i - 1)}\), \(\theta_{\mu \mu}^{(i - 2)}\), \(\theta_{\lambda \mu}^{(i)}\), \(\theta_{\lambda \mu}^{(i - 1)}\), \(\theta_{\mu \lambda}^{(i)}\), and \(\theta_{\mu \lambda}^{(i - 1)}\).
Partial derivatives become
\begin{align}
  \left. (\log p)_{\lambda} \right|_{\mu = 0} &= - \left(\frac{i x - j}{x - 1}\right) t\label{eq:gradient_lambda_zero_mu}\\
  \left. (\log p)_{\mu} \right|_{\mu = 0} &= \left(\frac{i x - j}{x - 1}\right) t - \frac{i (i - 1) x + j (j + 1) x^{-1} - 2 i j}{(i - j - 1) \lambda}\label{eq:gradient_mu_zero_mu}\\
  \left. (\log p)_{\lambda \lambda} \right|_{\mu = 0} &= \frac{(i - j) t^{2} x}{(x - 1)^{2}}\label{eq:hessian_lambda_lambda_zero_mu}\\
  \left. (\log p)_{\mu \mu} \right|_{\mu = 0} &= \frac{(i - j) t^{2} x}{(x - 1)^{2}} + \frac{i (i - 1) j (j + 1) (x - 1)^{4}}{(i - j - 1)^{2} (i - j - 2) \lambda^{2} x^{2}} + \nonumber\\
  &- \frac{i (i - 1) (x - 2 \lambda t) x + j (j + 1) (x^{-1} + 2 \lambda t) x^{-1} - 2 i j}{(i - j - 1) \lambda^{2}}\label{eq:hessian_mu_mu_zero_mu}\\
  \left. (\log p)_{\lambda \mu} \right|_{\mu = 0} &= - \frac{(i - j) t^{2} x}{(x - 1)^{2}} + \frac{i (i - 1) (1 - \lambda t) x + j (j + 1) (1 + \lambda t) x^{-1} - 2 i j}{(i - j - 1)^{2}\lambda^{2}}\label{eq:hessian_lambda_mu_zero_mu}\\
  \left. (\log p)_{\mu \lambda} \right|_{\mu = 0} &= \left. (\log p)_{\lambda \mu} \right|_{\mu = 0}\label{eq:hessian_mu_lambda_zero_mu}
\end{align}
where \(x = e^{\lambda t}\).
Note that equation (\ref{eq:gradient_mu_zero_mu}) has a discontinuity at the value \(j = i - 1\) where
\begin{equation*}
  \lim_{j \rightarrow (i - 1)^{-}} \left. (\log p)_{\mu} \right|_{\mu = 0} = - \infty, \qquad
  \lim_{j \rightarrow (i - 1)^{+}} \left. (\log p)_{\mu} \right|_{\mu = 0} = \infty
\end{equation*}
Equation (\ref{eq:hessian_mu_mu_zero_mu}) has a discontinuity at the value \(j = i - 2\) where
\begin{equation*}
  \lim_{j \rightarrow (i - 2)^{-}} \left. (\log p)_{\mu \mu} \right|_{\mu = 0} = \infty, \qquad
  \lim_{j \rightarrow (i - 2)^{+}} \left. (\log p)_{\mu \mu} \right|_{\mu = 0} = - \infty
\end{equation*}
Equations (\ref{eq:hessian_lambda_mu_zero_mu}) and (\ref{eq:hessian_mu_lambda_zero_mu}) have a discontinuity at the value \(j = i - 1\) where
\begin{equation*}
  \lim_{j \rightarrow (i - 1)^{-}} \left. (\log p)_{\lambda \mu} \right|_{\mu = 0} = - \infty, \qquad
  \lim_{j \rightarrow (i - 1)^{+}} \left. (\log p)_{\lambda \mu} \right|_{\mu = 0} = \infty
\end{equation*}

\subsubsection*{Birth rate is zero}
When \(\lambda\) approaches zero also the partial derivatives of \(\theta^{(h)}\), in general, approach zero.
Only exceptions are \(\theta_{\lambda}^{(j)}\), \(\theta_{\lambda}^{(j - 1)}\), \(\theta_{\mu}^{(j)}\), \(\theta_{\lambda \lambda}^{(j)}\), \(\theta_{\lambda \lambda}^{(j - 1)}\), \(\theta_{\lambda \lambda}^{(j - 2)}\), \(\theta_{\mu \mu}^{(j)}\), \(\theta_{\lambda \mu}^{(j)}\), \(\theta_{\lambda \mu}^{(j - 1)}\), \(\theta_{\mu \lambda}^{(j)}\), and \(\theta_{\mu \lambda}^{(j - 1)}\).
Partial derivatives become
\begin{align}
  \left. (\log p)_{\lambda} \right|_{\lambda = 0} &= \left(\frac{j x - i}{x - 1}\right) t - \frac{j (j - 1) x + i (i + 1) x^{-1} - 2 i j}{(j - i - 1) \mu}\label{eq:gradient_lambda_zero_lambda}\\
  \left. (\log p)_{\mu} \right|_{\lambda = 0} &= - \left(\frac{j x - i}{x - 1}\right) t\label{eq:gradient_mu_zero_lambda}\\
  \left. (\log p)_{\lambda \lambda} \right|_{\lambda = 0} &= \frac{(j - i) t^{2} x}{(x - 1)^{2}} + \frac{j (j - 1) i (i + 1) (x - 1)^{4}}{(j - i - 1)^{2} (j - i - 2) \mu^{2} x^{2}} + \nonumber\\
  &- \frac{j (j - 1) (x - 2 \mu t) x + i (i + 1) (x^{-1} + 2 \mu t) x^{-1} - 2 i j}{(j - i - 1) \mu^{2}} \label{eq:hessian_lambda_lambda_zero_lambda}\\
  \left. (\log p)_{\mu \mu} \right|_{\lambda = 0} &= \frac{(j - i) t^{2} x}{(x - 1)^{2}}\label{eq:hessian_mu_mu_zero_lambda}\\
  \left. (\log p)_{\lambda \mu} \right|_{\lambda = 0} &= - \frac{(j - i) t^{2} x}{(x - 1)^{2}} + \frac{j (j - 1) (1 - \mu t) x + i (i + 1) (1 + \mu t) x^{-1} - 2 i j}{(j - i - 1)^{2}\mu^{2}}\label{eq:hessian_lambda_mu_zero_lambda}\\
  \left. (\log p)_{\mu \lambda} \right|_{\lambda = 0} &= \left. (\log p)_{\lambda \mu} \right|_{\mu = 0}\label{eq:hessian_mu_lambda_zero_lambda}
\end{align}
where \(x = e^{\mu t}\).
Note that equation (\ref{eq:gradient_lambda_zero_lambda}) has a discontinuity at the value \(j = i + 1\) where
\begin{equation*}
  \lim_{j \rightarrow (i + 1)^{-}} \left. (\log p)_{\lambda} \right|_{\lambda = 0} = \infty, \qquad
  \lim_{j \rightarrow (i + 1)^{+}} \left. (\log p)_{\lambda} \right|_{\lambda = 0} = - \infty
\end{equation*}
Equation (\ref{eq:hessian_lambda_lambda_zero_lambda}) has a discontinuity at the value \(j = i + 2\) where
\begin{equation*}
  \lim_{j \rightarrow (i + 2)^{-}} \left. (\log p)_{\mu \mu} \right|_{\lambda = 0} = - \infty, \qquad
  \lim_{j \rightarrow (i + 2)^{+}} \left. (\log p)_{\mu \mu} \right|_{\lambda = 0} = \infty
\end{equation*}
Equations (\ref{eq:hessian_lambda_mu_zero_lambda}) and (\ref{eq:hessian_mu_lambda_zero_lambda}) have a discontinuity at the value \(j = i + 1\) where
\begin{equation*}
  \lim_{j \rightarrow (i + 1)^{-}} \left. (\log p)_{\lambda \mu} \right|_{\lambda = 0} = \infty, \qquad
  \lim_{j \rightarrow (i + 1)^{+}} \left. (\log p)_{\lambda \mu} \right|_{\lambda = 0} = - \infty
\end{equation*}

\subsubsection*{Both rates are zero}
The gradient of the log-transition probability at the origin is only defined when \(j = i\).
To prove it, we will compute the limit \((\lambda, \mu) \rightarrow (0, 0)\) from different directions and observe whether they all converge to the same value or not.
If \(j \neq i\)
\begin{equation*}
  \lim_{\lambda \rightarrow 0} \left. (\log p)_{\lambda} \right|_{\mu = 0} = \sign(j - i) \infty, \quad \lim_{\mu \rightarrow 0} \left. (\log p)_{\lambda} \right|_{\lambda = 0} = - \frac{(i + j) t}{2}, \quad \lim_{\lambda \rightarrow 0} \left. (\log p)_{\lambda} \right|_{\mu = \lambda} = 0
\end{equation*}
and
\begin{equation*}
  \lim_{\lambda \rightarrow 0} \left. (\log p)_{\mu} \right|_{\mu = 0} = - \frac{(i + j) t}{2}, \quad \lim_{\mu \rightarrow 0} \left. (\log p)_{\mu} \right|_{\lambda = 0} = \sign(i - j) \infty, \quad \lim_{\lambda \rightarrow 0} \left. (\log p)_{\mu} \right|_{\mu = \lambda} = 0
\end{equation*}
The same phenomenon can be observed with the second-order partial derivatives.
When \(j = i\) the first-order partial derivatives converge to \(- i t\).
Second-order partial derivatives \((\log p)_{\lambda \lambda}\) and \((\log p)_{\mu \mu}\) converge to 0 while \((\log p)_{\lambda \mu}\) and \((\log p)_{\mu \lambda}\) converge to \(i^{2} t^{2}\).
If we interpret the transition probability as the likelihood of a single time point observation, these results are intuitive.
Indeed, if \(j \neq i\) the rates cannot be both equal to zero.
If \(j = i\), instead, the hypothesis \(\lambda = \mu = 0\) is plausible because it is compatible with the observation.

\end{document}